\documentclass[draft]{myamsart}
\usepackage{myutf}
\RequirePackage{luatex85}
\usepackage[all]{xy}
\usepackage{wrapfig}
\numberwithin{equation}{section}

\def\paragraph#1{\subsubsection*{#1}}

\usepackage{comment}
\specialcomment{comment}{\par\medskip\noindent\begingroup\color{brown}\textsc{Note: }}{\endgroup\par\medskip}
\specialcomment{todo}{\par\begingroup\color{red}[[}{]]\endgroup\par} 
\includecomment{revised}
\excludecomment{final}
\includecomment{todo}
\excludecomment{comment}
\begin{revised}

\long\def\deleted#1{{\color{lightgray}\leavevmode\marginpar{\color{lightgray}\texttt{[deleted]}}#1}}
\end{revised}
\begin{final}

\long\def\deleted#1{}
\end{final}

\newcommand{\p}[1]{\mathaccent19{#1}} 

\def\seq 0->#1-#2->#3-#4->#5->0{0\longrightarrow
            #1\maparrow{#2}#3\maparrow{#4}#5\longrightarrow 0}
\newcommand{\maparrow}[1]{\mathrel{\mathop{\longrightarrow}\limits^{#1}}}
\renewcommand{\vectorfields}[2][]{\mathfrak{X}_{#1}(#2)}
\newcommand{\Hor}{\operatorname{Hor}}
\newcommand{\Ver}{\operatorname{Ver}}

\newcommand{\ver}[2][]{\Ver_{#1}(#2)}
\newcommand{\hlift}{ξ^\H} 
\newcommand{\vlift}{ξ^\V} 
\def\V{{\sssize\mathsf{V}}}
\def\C{{\sssize \mathsf{C}}}
\def\H{{\sssize\mathsf{H}}}
\renewcommand{\Sec}{\operatorname{Sec}}
\renewcommand{\sec}[2][]{\Sec_{#1}(#2)}

\newcommand{\Tplus}{+_{\sssize\mathsf{T}}}
\newcommand{\Tminus}{-_{\sssize\mathsf{T}}}
\newcommand{\Ttimes}{·_{\sssize\mathsf{T}}}
\newcommand{\caninv}[1]{{\raisebox{1.5pt}{$χ$}}_{#1}} 
\newcommand{\slit}[1]{\mathaccent23{#1}} 
\def\jvj{\mathsf{jvj}} 
\def\curv{\mathrm{Curv}}

\begin{document}

\title[On the linearization of a non linear connection]{On the construction of the linearization of a nonlinear connection}

\author[E.\ Mart\'{\i}nez]{Eduardo Mart\'{\i}nez}
\address{Eduardo Mart\'{\i}nez:
IUMA and Departamento de Matem\'{a}tica Aplicada,
Universidad de Zaragoza,
Pedro Cerbuna 12,
50009 Zaragoza, Spain}
\address{ORCID: https://orcid.org/0000-0003-3270-5681}
\email{emf@unizar.es}

\thanks{Partial financial support from MINECO (Spain) grant MTM2015-64166-C2-1-P and PGC2018-098265-B-C31, and from Gobierno de Arag\'on (Spain) grant E38\_17R is acknowledged}




\begin{abstract} 
The construction of a linear connection on a pullback bundle from a connection on a vector bundle is explained in terms of fiberwise linear approximation. This procedure clarifies the geometric meaning of the linearized connection as well as the associated parallel transport and curvature.  
\end{abstract}

\maketitle

\section{Introduction}

In a recent paper~\cite{linearization.1} we have studied the linerization of a non linear connection on a vector bundle. For a connection on a vector bundle $\map{π}{E}{M}$ we defined a linear connection on the pullback bundle $π^*E$ by means of a covariant derivative operator, expressed in terms of brackets of vector fields and other standard operations. Such a linearized connection is a very relevant object that has proven to be very useful in the solution of several problems related to dynamical systems defined by a system of second-order differential equations on a manifold, as they are the problem of characterizing the existence of coordinates in which the system is linear~\cite{linearizable.SODE}, the problem of decoupling a system of second-order differential equations into subsystems~\cite{decoupling, MikeFest, generalized.submersive}, the inverse problem of Lagrangian Mechanics~\cite{Douglas, IP.autonomo}, the classification of derivations~\cite{deriv,deriv2}, among others~\cite{tesis,linearization.1}. 

In a local coordinate system the coefficients of the linearized connection are the derivatives of the coefficients of the nonlinear connection with respect to the coordinates in the fibre. Its use dates back to Berwald in his studies on Finsler geometry, and was formalized and extended by Vilms~\cite{Vilms}, by using local charts. It was rediscovered many times, for instance in~\cite{tesis,linearizable.SODE}, and in many different geometrical versions~\cite{Massa.Pagani, Byrnes, Szilasi}. Despite the simplicity of the local description above, several attempts to describe geometrically the process of linearization in a proper, intrinsic and clear way have proven unsatisfactory or unsuccessful.

In our previous approach~\cite{linearization.1} we provided an intrinsic formula for the covariant derivative associated to the linear connection. In spite of the relevance and usefulness of the covariant derivative, the physical and geometrical meaning of the resultant operation is unclear. The purpose of the present paper is to clarify the meaning of such linearized connection.

For a nonlinear connection we will consider an appropriate restriction of the horizontal lift, obtaining a map between the fibers of two vector bundles. This map is then linearized by means of the differential at a given point, the best linear approximation of the map at that point. We will show that this procedure produces the horizontal lift of a linear connection on the pullback bundle. It will be shown that even in the case when the connection is smooth on a submanifold the resulting connection is linear, a fact which is fundamental for Finsler geometry, where the non linear connection is actually homogeneous and hence not smooth on the zero section. From the horizontal lift we will obtain the expression of the associated covariant derivative, obtaining the formula which served as the definition of the linearized connection in~\cite{linearization.1}. 

The above interpretation allows to obtain the horizontal lift of a vector field as the fibre derivative of such a vector field. As a consequence, the flow of the horizontal lift of a vector field is but the fibre derivative of the flow of the given vector field, and hence it provides a detailed description of the parallel transport system associated to the linearized connection, from where we will derive some explicit new formulas for its curvature.

A connection on a fiber bundle can be understood as a section of the first jet bundle of sections of that bundle. It is well known~\cite{NatOp} that by taking the fibre derivative (or the vertical derivative) of such section one obtains a linear connection which however does not coincide with the linearization defined in this paper, as it is defined on a different bundle. It will be shown that composing with a canonical map we obtain a section of the first jet bundle of the pullback bundle which coincides with the linearization. An equivalent version on the vertical bundle will also be provided.

The linearization of a nonlinear connection is a natural first order prolongation of the original connection. That is, their construction is invariant under fibered diffeomorphism, providing hence a geometric object, and uses only the first order derivatives of the initial connection. In the finite dimensional case, all natural first order prolongations of a connection were classified in~\cite{NatOp} obtaining a 1-parameter family of connections. The linearized connection is the only member of that family which is linear. We will provide a direct proof of this fact, which is valid for connections on Banach bundles. Moreover, it will be shown that the linearization can be also characterized as the only semibasic connection in the family, a property which is easier to work with in practical computations.  

\medskip

The paper is organized as follows. In Section~\ref{section:preliminaries} we will provide the necessary preliminary results. In Section~\ref{section:linearization} it will be shown that by linearizing the horizontal lifting map in an adequate sense we obtain a linear connection on the pullback bundle. In Section~\ref{section:natural.prolongation} we will show that all natural first order prolongations of a connection fit into a 1-parameter family of connections on the pullback bundle, and we characterize the linearization as the only element in the family which is linear. In Section~\ref{section:fibre.derivative} we will show the relation between the horizontal lift (with respect to the linearized connection) of a vector field and the fibre derivative of such vector field, and we will clarify the meaning of the parallel transport and of the curvature for the linearized connection. Finally, in Section~\ref{section:jets} we will reconsider the problem from the perspective of jet bundle theory, and we will construct the section associated to the linearized connection from the section associated to the original non linear one. 

\medskip

Throughout this paper a manifold will be understood as a smooth Banach manifold, and a fiber bundle will mean a smooth Banach fiber bundle (see~\cite{MTA, Lang}). All objects will be considered in the $C^∞$-smooth category. 

\section{Preliminaries}
\label{section:preliminaries}

The tangent bundle to a manifold $M$ will be denoted $\map{τ_M}{TM}{M}$, and the $\cinfty{M}$-module of vector fields on $M$ will be denoted $\vectorfields{M}$. For a fiber bundle $\map{τ}{P}{M}$, the set of sections of $τ$ will be denoted either by $\sec{τ}$ or by $\sec{P}$ whenever the projection is clear.

Given a map $\map{f}{N}{M}$ and a fiber bundle $\map{τ}{P}{M}$, the pullback bundle $\map{f^*τ}{f^*P}{N}$ has total space $f^*P=\set{(n,p)∈ N\times P}{τ(p)=f(n)}$ and projection $f^*τ(n,p)=n$. A section of $\map{f^*τ}{f^*P}{N}$ is equivalent to a section of $P$ along $f$, that is, a map $\map{\sigma}{N}{P}$ such that $τ∘σ=f$. The set of sections along $f$ will be denoted by $\sec[f]{τ}$ or $\sec[f]{P}$. For the details see~\cite{Poor, sections.along.maps}. 

We will mainly be concerned with the case of a vector bundle $\map{π}{E}{M}$. The pullback $\map{f^*π}{f^*E}{N}$ is also a vector bundle. More concretely, we are interested in the case where the map $f$ is just the projection $π$. A section $\sigma∈\sec{π^*E}$ is said to be basic if there exists a section $\alpha∈\sec{E}$ such that $σ=α∘π$. We will frequently identify both sections, i.e. $\alpha$ indicates both the section of $E$ and the section of $π^*E$, and the context will make clear the meaning. However, in the last section we will need to be more precise.  

The vertical subbundle $\map{τ_E^\V\equivτ_E\big|_{\Ver{π}}}{\Ver{π}=\Ker(Tπ)}{E}$ is canonically isomorphic to the pullback bundle $\map{π^*π}{π^*E}{E}$. The isomorphism is the vertical lifting map $\map{\vlift_π}{π^*E}{\ver{π}⊂TE}$, defined by $\vlift_π(a,b)=\frac{d}{ds}(a+sb)|_{s=0}$ for every $(a,b)∈π^*E$. The vertical projection is the map $\map{ν_π}{\Ver{π}}{E}$ given by $ν_π=\pr_2○(\vlift_π)^{-1}$, that is $ν_π\bigl(\vlift_π(a,b)\bigr)=b$ for $(a,b)∈π^*E$. For a section $\sigma∈\sec{π^*E}$ the vertical lift of $\sigma$ is the vector field $\sigma^\V∈\vectorfields{E}$ given by $σ^\V=\vlift_π∘\sigma$. For a section $σ$ of $E$ along $π$ we write $σ^\V$ for the vertical lift of the section of $π^*E$ associated to $σ$. Clearly we have $ν_π∘σ^\V=σ$ for $σ∈\sec[π]{E}$

The tangent bundle $\map{τ_{π^*E}}{T(π^*E)}{π^*E}$ to the manifold $π^*E$ can be canonically identified with the pullback bundle $(Tπ)^*(TE)$. The identification is given by $(ω_1,ω_2)^{\tsize\cdot}(0)\simeq\bigl(\dot{ω}_1(0),\dot{ω}_2(0)\bigr)$ for $ω_1,ω_2$ curves in $E$ such that $π∘ω_1=π∘ω_2$. Along this paper we will use such identification $ T(π^*E)\simeq(Tπ)^*(TE)$, and therefore a tangent vector to $π^*E$ at the point $(a,b)$ will be considered as a pair $(w_1,w_2)∈T_aE\times T_bE$ such that $Tπ(w_1)=Tπ(w_2)$.

\subsection*{Connections}
For a vector bundle $\map{π}{E}{M}$ there is a short exact sequence of vector bundles over $E$
\[
\seq 0-> π^*E -\vlift_π-> TE -p_E-> π^*TM ->0,
\]
where $\map{p_E}{TE}{π^*(TM)}$ is the projection $p_E(w)=(τ_E(w),Tπ(w))$. This sequence is known as the fundamental sequence of $\map{π}{E}{M}$.

An Ehresmann  connection on the vector bundle $E$ is a (right) splitting of the fundamental sequence, that is, a map $\map{\xi^\H}{π^*TM}{TE}$ such that $p_E∘\xi^\H=\id_{π^*(TM)}$. The map $\hlift$ is said to be the horizontal lifting. Equivalently, a connection is given by the associated left splitting of that sequence, that is, the map $\map{\vartheta}{TE}{π^*E}$ such that $\vartheta∘\vlift_π=\id_{π^*E}$ and $\vartheta∘\hlift=0$. Closely related to the map $\vartheta$ is the the connector $\map{κ=\pr_2∘\vartheta}{TE}{E}$, also known as the Dombrowski connection map. See~\cite{Poor} for the general theory of connections on vector bundles.

A connection decomposes $TE$ as a direct sum $TE=\Hor\oplus\Ver E$, where $\Hor=\Im(\hlift)=\Ker(\vartheta)$ is said to be the horizontal subbundle. The projectors over the horizontal and vertical subbundles are given by $P_\H=\hlift∘ p_E$ and $P_\V=\vlift_π∘\vartheta$. For a section $X∈\sec{π^*TM}$ the horizontal lift of $X$ is the vector field $X^\H∈\vectorfields{E}$ given by $X^\H=\hlift∘X$. For a vector field $X∈\sec[π]{TM}$ along $π$ the symbol $X^\H$ denotes the horizontal lift of the associated section of $π^*(TM)$. The curvature of the non linear connection is the $E$-valued 2-form $\map{R}{π^*(TM\wedge TM)}{π^*E}$ defined by
\[
R(X,Y)=\vartheta([X^\H,Y^\H]),\qquad\qquad X,Y∈\sec{π^*TM}.
\]

\medskip

A connection $\hlift$ on the pullback bundle $\map{π^*π}{π^*E}{E}$ is therefore given by a map $\map{\hlift}{(π^*π)^*(TE)}{T(π^*E)}$ of the form $\hlift\bigl((a,b),w\bigr)=\bigl(w,B(a,b)w\bigr)$ for $\map{B(a,b)}{T_aE}{T_bE}$ a linear map such that $T_bπ=B(a,b)∘T_aπ$.

A connection $\hlift$ on $\map{π^*π}{π^*E}{E}$ is said to be basic if it is the pullback of a connection $ξ^{\ul{\H}}$ on $\map{π}{E}{M}$, i.e. it is of the form $\hlift((a,b),w)=(w,ξ^{\ul{\H}}(b,Tπ(w)))$ for all $(a,b)∈π^*E$ and $w∈T_aE$. The connection $\hlift$ is said to be semibasic if for all $(a,b)∈π^*E$ the horizontal lift of a vertical vector $w∈\ver[a]{π}$ is of the form $\hlift((a,b),w)=(w,0_b)$. Obviously a basic connection is semibasic but the converse does not hold. In terms of the connector, the connection is semibasic if $κ(w,0_b)=0_{π(b)}$ for every vertical $w∈\ver[a]{π}$.

\subsection*{Double vector bundle structure}
For a vector bundle $\map{π}{E}{M}$ the manifold $TE$ has two different vector bundle structures. On one hand we have the standard tangent bundle $\map{τ_E}{TE}{E}$ with operations denoted by $+$ and $·$ (the dot usually omitted in the notation). On the other, we also have a vector bundle structure $\map{Tπ}{TE}{TM}$ where the operations are the differential of the operations in $\map{π}{E}{M}$ and which will be denoted $\Tplus$ and $\Ttimes$. They can be easily defined in terms of vectors tangent to curves by
\[
λ\Ttimes \frac{dα}{dt}(0)=\frac{d(λα)}{dt}(0)
\]
and
\[
\frac{dα}{dt}(0)\Tplus\frac{dβ}{dt}(0)=\frac{d(α+β)}{dt}(0),
\]
where the curves $α$ and $β$ satisfy $π∘α=π∘β$.

These operations are compatible, in the sense that they define a structure of double vector bundle on $TE$ over $E$ and $TM$. We will list in the next paragraphs a few properties, which are consequences of such a double structure, and that will be needed later on. For more details on double vector bundles see~\cite{Mackenzie, GrRo}.

The two different operations of addition are compatible in the sense that they satisfy the \textsl{interchange law}
\[
(w_1+w_2)\Tplus(w_3+w_4)=(w_1\Tplus w_3)+(w_2\Tplus w_4),
\]
for $w_1,w_2,w_3,w_4∈TE$ such that $τ_{E}(w_1)=τ_{E}(w_2)$, $τ_{E}(w_3)=τ_{E}(w_4)$ and $Tπ(w_1)=Tπ(w_3)$, $Tπ(w_2)=Tπ(w_4)$. Moreover, we have the following properties, which apply when the operations on the left hand side are defined,
\[
\begin{aligned}
&τ_E(v_1+v_2)=τ_E(v_1)=τ_E(v_2)\\
&τ_E(λv)=τ_E(v)\\[7pt]
&τ_E(v_1\Tplus v_2)=τ_E(v_1)+τ_E(v_2)\\
&τ_E(λ\Ttimes v)=λτ_E(v)
\end{aligned}
\qquad\qquad
\begin{aligned}
&Tπ(v_1+v_2)=Tπ(v_1)+Tπ(v_2)\\
&Tπ(λv)=λTπ(v)\\[7pt]
&Tπ(v_1\Tplus v_2)=Tπ(v_1)=Tπ(v_2)\\
&Tπ(λ\Ttimes v)=Tπ(v).
\end{aligned}
\]

\medskip

A connection $\map{\hlift}{π^*TM}{TE}$ on the vector bundle $\map{π}{E}{M}$ is said to be linear if $h(a,v)$ is linear in the argument $a$ for every fixed $v∈TM$, that is, if it satisfies
\begin{gather*}
\hlift(a+a',v)=\hlift(a,v)\Tplus\hlift(a',v)\\
\hlift(λa,v)=λ\Ttimes(\hlift(a,v)),
\end{gather*}
for every $λ∈ℝ$, every $(a,v)∈π^*TM$ and $a'∈E$ with $π(a)=π(a')$. A linear connection can be equivalently described by means of a covariant derivative operator $\map{D}{\vectorfields{M}\times\sec{E}}{\sec{E}}$. The relation between the covariant derivative~$D$ and the connection $\hlift$ is given by $D_vσ=ν_π\bigl(Tσ(v)-\hlift(σ(m),v)\bigr)=κ(Tσ(v))$ for $v∈T_mM$ and $σ∈\sec{E}$.

\medskip

Both $\map{τ_E}{TE}{E}$ and $\map{Tπ}{TE}{TM}$ being vector bundles have corresponding associated vertical lifting maps to $T(TE)$. On one hand we have the map $\map{\vlift_{τ_E}}{τ_E^*(TE)}{\ver{τ_E}⊂TTE}$ given by
\[
\vlift_{τ_E}(v,w)=\frac{d}{ds}\bigl(v+sw\bigr)\at{s=0},\qquad\text{where $τ_E(v)=τ_E(w)$,}
\]
and on the other $\map{\vlift_{Tπ}}{(Tπ)^*(TE)}{\ver{Tπ}⊂TE}$ given by
\[
\vlift_{Tπ}(v,w)=\frac{d}{ds}\bigl(v\Tplus s\Ttimes w\bigr)\at{s=0},\qquad\text{where $Tπ(v)=Tπ(w)$.}
\]

Correspondingly, we have two vertical projection maps, $\map{ν_{τ_E}}{\ver{τ_E}}{TE}$ defined by $ν_{τ_E}(\vlift_{τ_E}(v,w))=w$, and $\map{ν_{Tπ}}{\ver{Tπ}}{TE}$ defined by $ν_{Tπ}(\vlift_{Tπ}(v,w))=w$. For the constructions in this paper the most relevant one is $ν_{Tπ}$, which has the following properties,
\begin{align*}
&ν_{Tπ}(v_1\Tplus v_2)=ν_{Tπ}(v_1)+ν_{Tπ}(v_2)\\
&ν_{Tπ}(λ\Ttimes v)=λν_{Tπ}(v)\\
&ν_{Tπ}(λv)=λ\Ttimes ν_{Tπ}(v).
\end{align*}

\subsection*{Some auxiliary results}
We will need the following auxiliary properties which can be found scattered in the literature.

\begin{proposition}
\label{several.relations}
The following properties hold.
\begin{enumerate}
\item $τ_E○ν_{Tπ}=ν_π○Tτ_E\big|_{\ver{Tπ}}$.
\item For $v,w∈TE$ such that $Tπ(v)=Tπ(w)$,
\[
Tτ_E\bigl(\vlift_{Tπ}(v,w)\bigr)=\vlift_{τ_E}\bigl(τ_E(v),τ_E(w)\bigr).
\]
\item For $v,w∈\ver{π}$ such that $τ_E(v)=τ_E(w)$,
\[
Tν_π\bigl(\vlift_{τ_E}(v,w)\bigr)=\vlift_π\bigl(ν_π(v),ν_π(w)\bigr).
\]
\end{enumerate}
\end{proposition}

\begin{proof}
\thetag{2} If $v,w∈\ver{π}$ are such that $Tπ(v)=Tπ(w)$ then
\begin{align*}
Tτ_E\bigl(\vlift_{Tπ}(v,w)\bigr)
&=Tτ_E\Bigl(\frac{d}{ds}(v\Tplus s\Ttimes w)\at{s=0}\Bigr)\\
&=\frac{d}{ds}τ_E\bigl(v\Tplus s\Ttimes w)\bigr)\at{s=0}\\
&=\frac{d}{ds}\bigl(τ_E(v)+sτ_E(w)\bigr)\at{s=0}\\
&=\vlift_{τ_E}\bigl(τ_E(v),τ_E(w)\bigr).
\end{align*}

\thetag{1} We write $V∈\ver[v]{Tπ}$ in the form $V=\vlift_{Tπ}(v,w)$ for a unique $w∈TE$ such that $Tπ(v)=Tπ(w)$. Then, using \thetag{2} we have
\[
ν_π○Tτ_E(V)
=ν_π○Tτ_E(\vlift_{Tπ}(v,w))\\
=ν_π\bigl(\vlift_{τ_E}\bigl(τ_E(v),τ_E(w)\bigr)\\
=τ_E(w)
\]
and on the other hand
\[
τ_E○ν_{Tπ}(V)=τ_E○ν_{Tπ}\Bigl(\frac{d}{ds}(v\Tplus s\Ttimes w)\at{s=0}\Bigr)=τ_E(w).
\]

\thetag{3} If $v,w∈\ver[a]{π}$ we can write $v=\vlift_π(a,b)$ and $w=\vlift_π(a,c)$ for some $b,c∈E_{π(a)}$. Then
\[
\vlift_{τ_E}(v,w)
=\frac{d}{ds}[\vlift_π(a,b)+s\vlift_π(a,c)]\at{s=0}
=\frac{d}{ds}\vlift_π(a,b+sc)\at{s=0},
\]
and hence
\begin{align*}
Tν_π\bigl(\vlift_{τ_E}(v,w)\bigr)
&=Tν_π\Bigl(\frac{d}{ds}\vlift_π(a,b+sc)\at{s=0}\Bigr)\\
&=\frac{d}{ds}ν_π\bigl(\vlift_E(a,b+sc)\bigr)\at{s=0}\\
&=\frac{d}{ds}(b+sc)\at{s=0}\\
&=\vlift_π(b,c).
\end{align*}
This proves the third property, as $ν_π(v)=b$ and $ν_π(w)=c$.
\end{proof}

\medskip 

On the double tangent bundle to a manifold $P$ we have a canonical involution $\map{\caninv{TP}}{T(TP)}{T(TP)}$ which intertwines the two vector bundle structures $\map{τ_{TP}}{TTP}{TP}$ and $\map{Tτ_P}{TTP}{TP}$, i.e. $\caninv{TP}○\caninv{TP}=\id_{TTP}$, $τ_{TP}○\caninv{TP}=Tτ_P$. It is easily defined in terms of 1-parameter families of curves $\map{θ}{ℝ^2}{P}$ by
\[
\caninv{TP}\Bigl(\frac{d}{ds}\Bigl(\frac{d}{dt}θ(s,t)\at{t=0}\Bigr)\at{s=0}\Bigr)
=\frac{d}{dt}\Bigl(\frac{d}{ds}θ(s,t)\at{s=0}\Bigr)\at{t=0}.
\] 
It satisfies 
\[
\begin{aligned}
&\caninv{TP}(v_1+v_2)=\caninv{TP}(v_1)\Tplus\caninv{TP}(v_2)\\
&\caninv{TP}(v_1\Tplus v_2)=\caninv{TP}(v_1)+\caninv{TP}(v_2)
\end{aligned}
\qquad\qquad
\begin{aligned}
&\caninv{TP}(λ\Ttimes v)=λ\caninv{TP}(v)\\
&\caninv{TP}(λv)=λ\Ttimes\caninv{TP}(v).
\end{aligned}
\]

In the concrete case $P=E$ we also have the following properties that will be used later on.
\begin{proposition}
\label{involution.vertical}
Let $\map{π}{E}{M}$ be a fibre bundle. The canonical involution $\map{\caninv{TE}}{TTE}{TTE}$ satisfies
\begin{enumerate}
\item $\vlift_{Tπ}=\caninv{TE}∘T\vlift_π$.
\item $ν_{Tπ}=Tν_π∘\caninv{TE}|_{\ver{Tπ}}$.
\item $\caninv{TE}$ restricts to a diffeomorphism $\map{\caninv{TE}}{\ver{Tπ}}{T\ver{π}}$.
\end{enumerate}
\end{proposition}
\begin{proof}
\thetag{1} Given $(v,w)∈(Tπ)^*TE\simeq T(π^*E)$ we consider a pair of curves $α(t),β(t)$ in $E$ such that $π∘α=π∘β$ and $v=\dot{α}(0)$, $w=\dot{β}(0)$. Then
\begin{align*}
\caninv{TE}∘T\vlift_π(v,w)
&=\caninv{TE}\Bigl(\frac{d}{dt}\,\vlift_π\bigl(α(t),β(t)\bigr)\at{t=0}\Bigr)\\
&=\caninv{TE}\Bigl(\frac{d}{dt}\frac{d}{ds}\bigl(α(t)+sβ(t)\bigr)\at{s=0}\at{t=0}\Bigr)\\
&=\frac{d}{ds}\frac{d}{dt}\bigl(α(t)+sβ(t)\bigr)\at{t=0}\at{s=0}\\
&=\frac{d}{ds}\bigl(v\Tplus s\Ttimes w\bigr)\at{s=0}\\
&=\vlift_{Tπ}(v,w).
\end{align*}

\thetag{2} Taking inverses in the above equation we get $(\vlift_{Tπ})^{-1}=(T\vlift_π)^{-1}∘\caninv{TE}$, and taking the second component we obtain $ν_{Tπ}=Tν_π∘\caninv{TE}|_{\ver{Tπ}}$. 

\thetag{3} The third property follows from the first and the fact that $\vlift_{Tπ}$ and $\vlift_π$ are isomorphisms.
\end{proof}

\section{Linearization of a non linear connection on a vector bundle}
\label{section:linearization}

Consider a nonlinear connection $\Hor$ on a vector bundle $\map{π}{E}{M}$ given by a horizontal lifting map $\map{\hlift}{π^*TM}{TE}$. Fix a point $a∈E$ and a vector $w∈T_aE$ tangent to $E$ at the point $a$. Setting $m=π(a)$ and $\bar{w}=Tπ(w)∈T_mM$ we consider the map $c∈E_m\mapsto \hlift(c,\bar{w})∈TE$. The image of this map is in the fibre $(Tπ)^{-1}(\bar{w})$ because $Tπ(\hlift(c,\bar{w}))=\bar{w}$ for all $c∈E_m$. Therefore we have a map between two vector spaces
\[
\map{ψ_{\bar{w}}}{π^{-1}(m)}{(Tπ)^{-1}(\bar{w})},\qquad c\mapsto \hlift(c,\bar{w}).
\]
The connection $\hlift$ is linear if and only if the map $ψ_{\bar{w}}$ is linear for every fixed~$\bar{w}$. When $\hlift$ is not linear, the differential of the map $ψ_{\bar{w}}$ at the point $a$, the best linear approximation to $ψ_{\bar{w}}$ at $a$, is a linear map $b\mapsto Dψ_v(a)b$ between the same vector spaces. The vector $Dψ_{\bar{w}}(a)b$ is thus an element of $(Tπ)^{-1}(\bar{w})⊂TE$. 

One has to be careful with the differential $Dψ_{\bar{w}}(a)b$ because the operations (sum and product by scalars) involved in this limit are the ones in the vector space $(Tπ)^{-1}(v)$, i.e. those of the vector bundle $\map{Tπ}{TE}{TM}$. In other words the precise meaning of the differential $Dψ_{\bar{w}}(a)b$ is
\begin{align*}
Dψ_{\bar{w}}(a)b&=\lim_{s→0}\frac{1}{s}\Ttimes[ψ_{\bar{w}}(a+sb) \Tminus ψ_{\bar{w}}(a)]
\\
&=\lim_{s→0}\frac{1}{s}\Ttimes[\hlift(a+sb,\bar{w}) \Tminus \hlift(a,\bar{w})].
\end{align*}

The vector $Dψ_{\bar{w}}(a)b$ is tangent to $E$ at the point $b$. In this way, for $(a,b)∈π^*E$ we have defined a linear  map $B(a,b):w∈T_aE\mapsto Dψ_{\bar{w}}(a)b∈T_bE$ which satisfies $T_bπ(B(a,b)w)=T_aπ(w)$, and hence it defines a horizontal lift $\bigl((a,b),w\bigr)\mapsto \bigl(w,B(a,b)w\bigr)$ on the pullback bundle $π^*E→E$, which will be shown to be a linear connection on such bundle.

To avoid as much as possible the explicit use of limits, instead of working with the differential in the above vector spaces we will make use of the canonical identification of a vector tangent to the fibre with a vertical vector. That is, we will consider the vector tangent to the curve $s\mapsto \hlift(a+sb,\bar{w})$, which is a $Tπ$-vertical vector, and then we take the element in the fibre corresponding to it. In this way, we can equivalently define
\[
B(a,b)w=ν_{Tπ}\Bigl(\,\frac{d}{ds}\hlift(a+sb,Tπ(w))\at{s=0}\,\Bigr),
\]
which is the expression to be used in what follows.

\begin{theorem}
\label{linear.connection.0}
Let $\Hor$ be an Ehresmann connection on a vector bundle $\map{π}{E}{M}$ with horizontal lift $\hlift$. For $(a,b)∈π^*E$ and $w∈T_aE$ let $B(a,b)w∈TE$ be the vector  
\[
B(a,b)w = ν_{Tπ}\Bigl(\,\frac{d}{ds}\hlift(a+sb,Tπ(w))\at{s=0}\,\Bigr).
\]
Then the map $\map{ξ^{\bar{\H}}}{π^*E\times_E TE}{T(π^*E)}$ given by $ξ^{\bar{\H}}\bigl((a,b),w\bigr) = \bigl(w,B(a,b)w\bigr)$ is the horizontal lift of a linear connection $\ol{\Hor}$ on the pullback  vector bundle $\map{π^*π}{π^*E}{E}$.
\end{theorem}
\begin{proof}
Given $(a,b)∈π^*E$ and $w∈T_aE$ we consider the curve $\map{α}{ℝ}{TE}$ given by $α(s)=\hlift\bigl(a+sb,Tπ(w)\bigr)$. In terms of this curve, the vector $B(a,b)w$ is 
\[
B(a,b)w=ν_{Tπ}(\dot{α}(0)).
\]
The curve $α$ depends on $a$, $b$, and $w$. At every step in the proof, when needed, we will indicate with a subscript the dependence that is relevant for the argumentation.
\begin{enumerate}
\item The map $B$ is well defined and $B(a,b)w∈T_bE$. Indeed, from $(Tπ○α)(s)=Tπ(w)$, constant, we get that $TTπ(\dot{α}(0))=0$. Thus $\dot{α}(0)∈\ver{Tπ}$ and we can apply $ν_{Tπ}$ to it. On the other hand, using Proposition~\ref{several.relations} \thetag{1} we get
\begin{align*}
τ_E\bigl(B(a,b)w\bigr)
&=τ_E○ν_{Tπ}(\dot{α}(0))\\
&=ν_π○Tτ_E(\dot{α}(0))\\
&=ν_π\Bigl(\frac{d}{ds}(τ_E○α)(0)\Bigr)\\
&=ν_π\Bigl(\frac{d}{ds}(a+sb)\at{s=0}\Bigr)\\
&=b.
\end{align*}

\item $B(a,b)w$ is linear in $w$. Indeed, we have that $\hlift(a,λv)=λ\hlift(a,v)$, from where $α_w(s)=\hlift(a+sb,Tπ(w))$ satisfies $α_{λw}(t)=λα_w(t)$ and hence $\dot{α}_{λw}(0)=λ\Ttimes\dot{α}_w(0)$. Therefore 
\begin{align*}
B(a,b)(λw)
&=ν_{Tπ}(\dot{α}_{λw}(0))\\
&=ν_{Tπ}(λ\Ttimes\dot{α}_{w}(0))\\
&=λν_{Tπ}(\dot{α}_{w}(0))\\
&=λB(a,b)w.
\end{align*}
Thus $B(a,b)$ is a linear map from $T_aE$ to $T_bE$.

\item $\bigl(w,B(a,b)w\bigr)∈ T_{(a,b)}π^*E$. Indeed, 
\begin{align*}
Tπ(B(a,b)w)
&=Tπ○ν_{Tπ}(\dot{α}(0))\\
&=Tπ○τ_{TE}(\dot{α}(0))\\
&=Tπ(α(0))\\
&=Tπ(w).
\end{align*}
\end{enumerate}

\noindent By \thetag{1}, \thetag{2} and \thetag{3} we get that $\bigl((a,b),w\bigr)\mapsto \bigl(w,B(a,b)w\bigr)$ is the horizontal lift of a connection $\ol{\Hor}$ on $\map{π^*π}{π^*E}{E}$.

\begin{enumerate}
\setcounter{enumi}{3}
\item The connection $\ol{\Hor}$ is linear. We have to prove linearity of $B(a,b)w$ in~$b$. In view of smoothness of $B$ we just need to prove homogeneity, that is, $B(a,λb)w=λ\Ttimes B(a,b)w$ for all $λ∈ℝ$. The curve $α_b(s)=\hlift(a+sb,Tπ(w))$ satisfies $α_{λb}(s)=α_b(λs)$, from where $\dot{α}_{λb}(0)=λ\dot{α}_b(0)$. Hence
\begin{align*}
B(a,λb)w
&=ν_{Tπ}(\dot{α}_{λb}(0))\\
&=ν_{Tπ}(λ\dot{α}_{b}(0))\\
&=λ\Ttimesν_{Tπ}(\dot{α}_{b}(0))\\
&=λ\Ttimes B(a,b)w.
\end{align*}
\end{enumerate}
This ends the proof
\end{proof}

The connection $\ol{\Hor}$ is said to be the linearization of the non linear connection $\Hor$. It is a semibasic linear connection on $π^*E$. Indeed, if $(a,b)∈π^*E$ and $w∈\ver[a]{π}$ then $B(a,b)w=0_b$, so that $ξ^{\bar{\H}}((a,b),w)=(w,0_b)$. 

If $\Hor$ is a linear connection, then the linearization $\ol{\Hor}$ is basic and equal to the pullback of $\Hor$. Indeed, if $(a,b)∈π^*E$ and $w∈T_aE$ then 
\begin{align*}
B(a,b)w
&=ν_{Tπ}\Bigl(\frac{d}{ds}\hlift(a+sb,Tπ(w))\at{s=0}\Bigr)\\
&=ν_{Tπ}\Bigl(\frac{d}{ds}\bigl[\hlift\bigl(a,Tπ(w)\bigr)\Tplus s\Ttimes \hlift\bigr(a,Tπ(w)\bigl)\bigr]\at{s=0}\Bigr)\\
&=ν_{Tπ}\bigl(\vlift_{Tπ}(\hlift(a,Tπ(w)),\hlift(b,Tπ(w)))\bigr)\\
&=\hlift(b,Tπ(w)).
\end{align*}
Therefore $ξ^{\bar{\H}}((a,b),w)=(w,\hlift(b,Tπ(w)))$ and hence $\ol{\Hor}$ is the pullback of $\Hor$.

\paragraph{Coordinate expressions}
Asume that $E$ is finite dimensional. Consider local coordinates $(x^i)$ defined on an open subset $\calu⊂M$ and a local basis $\{e_A\}$ of local sections of $E$ defined on $\calu$. If a point $m∈\calu$ has coordinates $(x^i)$ and the components of a vector $a∈E_m$ in the basis $\{e_A(m)\}$ are $(y^A)$ (i.e. $a=y^Ae_A(m)$) then the coordinates of $a$ are $(x^i,y^A)$. 

Let us find the expression of the map $B$ in local coordinates as above. If the point $a∈E$ has coordinates $(x,y)$, the point $b∈E$ has coordinates $(x,z)$, and the  the tangent vector $w∈T_aE$ has components $(w^i,w^A)$ in the coordinate basis $\{\partial/\partial x^i,\partial/\partial y^A\}$, then the coordinate expression of the curve $α(s)=\hlift(a+sv,Tπ(w))$ is
\[
α(s)=\bigl(x^i,y^A+sz^A,\ w^i,-Γ^A_i(x,y+sz)w^i\bigr).
\]
Taking the derivative at $s=0$ we get
\begin{align*}
\dot{α}(0)=\bigl(x^i,y^A,w^i,-Γ^A_i(x,y)w^i,0,z^A,0,-Γ^A_{iB}(x,y)z^Bw^i\bigr)
\end{align*}
where $Γ^A_{iB}(x,y)=\partial Γ^A_i/\partial y^B(x,y)$. Finally applying $ν_{Tπ}$ we get
\[
B(a,b)w=\bigl(x^i,z^A,w^i,-Γ^A_{iB}(x,y)z^Bw^i\bigr).
\]
This expression shows clearly that $B(a,b)w$ is linear both in $b$ and $w$.

An equivalent expression is 
\[
B(a,b)
\Bigl(w^i\pd{}{x^i}\at{(x,y)}+w^A\pd{}{y^A}\at{(x,y)}\Bigr)
=
w^i\Bigl(\pd{}{x^i}\at{(x,z)}-Γ^A_{iB}(x,y)z^B\pd{}{y^A}\at{(x,z)}\Bigr),
\]
and thus the horizontal lift is given by
\begin{align*}
ξ^{\bar{\H}}&\Bigl((a,b), \Bigl(w^i\pd{}{x^i}\at{(x,y)}+w^A\pd{}{y^A}\at{(x,y)}\Bigr)\Bigr)=\\
&=
\biggl(
w^i\pd{}{x^i}\at{(x,y)}+w^A\pd{}{y^A}\at{(x,y)}\,,\, w^i\pd{}{x^i}\at{(x,z)}-Γ^A_{iB}(x,y)z^Bw^i\pd{}{y^A}\at{(x,z)}
\biggr).
\end{align*}
A local basis of the horizontal distribution is 
\begin{align*}
\bar{H}_i
&=
\biggl(
\pd{}{x^i}\at{(x,y)}
\,,\,
\pd{}{x^i}\at{(x,z)}-Γ^A_{iB}(x,y)z^B\pd{}{y^A}\at{(x,z)}
\biggr),
\\
\bar{H}_A
&=
\biggl(
\pd{}{y^A}\at{(x,y)}
\,,\,
0_{(x,z)}
\biggr).
\end{align*}
Together with $\bar{V}_A=(0_{(x,y)},\partial/\partial z^A\big|_{(x,z)})$ they form a local basis of vector fields on the manifold $π^*E$.

In connection theory it is customary to indicate locally a connection as the Pfaff system given by the annihilator of the horizontal distribution. For the original connection $\Hor$ such a Pfaff system is $dy^A+Γ^A_i(x,y)dx^i=0$. For the linearized connection $\ol{\Hor}$ it is given by $dz^A+Γ^A_{iB}(x,y)z^Bdx^i=0$. Thus, symbolically, the linearization procedure consist in
\[
\Hor:\ \Bigl\{dy^A+Γ^A_i(x,y)dx^i=0
\quad\rightsquigarrow\quad
\ol{\Hor}:\ \Bigl\{dz^A+Γ^A_{iB}(x,y)z^Bdx^i=0.
\]
where we recall that $\dsize Γ^A_{iB}=\pd{Γ^A_i}{y^B}$.

\subsection*{Restricted domain}
In many situations the initial non linear connection is not defined (or is not smooth) on the whole vector bundle~$E$. A typical example occurs in Finsler geometry where $E=TM$ and the connection is defined by a Finsler metric. The most notable Finsler connections are homogeneous (generally non linear) and are smooth on the slit tangent bundle $\map{\slit{τ}_M}{\slit{T}M}{M}$, i.e. on the tangent bundle to the base manifold without the zero section.

In this respect, at step \thetag{4} in the proof of Theorem~\ref{linear.connection.0}, we have assumed smoothness of $B(a,b)$ at $b=0$ and hence linearity followed from homogeneity. In the case of a restricted domain (not including the zero section) Theorem~\ref{linear.connection.0} only proves homogeneity of the `linearized' connection $\ol{\Hor}$. However, as it is clear from the coordinate expressions above, the connection coefficients are indeed linear in the variable $z$. We will next prove that the connection  $\ol{\Hor}$ is in fact linear even in this more general case.

\begin{theorem}
Let $C⊂E$ be a submanifold of $E$ such that $π^{-1}(m)∩C$ is a non-empty open subset of $E_m$ for every $m∈M$, and denote by $\map{π_C}{C}{M}$ the restriction of $π$ to $C$. Given a nonlinear connection $\Hor$ on the bundle $\map{π}{E}{M}$, smooth in the submanifold $C$, the map $B$ defined as in Theorem~\ref{linear.connection.0} provides a linear connection on the pullback bundle $π^*E$ with domain $π_C^*E=\set{(a,b)∈π^*E}{a∈C}$.     
\end{theorem}
\begin{proof}
It can be easily checked that all the steps in the proof of Theorem~\ref{linear.connection.0} remains valid if the domain of $\Hor$ is the submanifold $C$ of $E$. To prove linearity we just need to show that the family $B$ satisfies the property $B(a,b_1+b_2)w=B(a,b_1)w\Tplus B(a,b_2)w$.

Consider the curve $α_b(s)=\hlift(a+sb,Tπ(w))$ as above, and the map $A(s_1,s_2)=\hlift(a+s_1b_1+s_2b_2,Tπ(w))$. Then $A(s,s)=α_{b_1+b_2}(s)$, $A(s,0)=α_{b_1}(s)$ and $A(0,s)=α_{b_2}(s)$, so that 
\begin{align*}
B(a,b_1+b_2)w
&=ν_{Tπ}\left(\frac{d}{ds}α_{b_1+b_2}(s)\at{s=0}\right)\\
&=ν_{Tπ}\left(\frac{d}{ds}A(s,s)\at{s=0}\right)\\
&=ν_{Tπ}\left(\pd{A}{s_1}(0,0)+\pd{A}{s_2}(0,0)\right)\\
&=ν_{Tπ}\left(\frac{d}{ds}α_{b_1}(s)\at{s=0}+\frac{d}{ds}α_{b_2}(s)\at{s=0}\right)\\
&=B(a,b_1)w\Tplus B(a,b_2)w.
\end{align*}
Therefore the connection defined by $B$ is not only homogeneous but linear.
\end{proof}

As a consequence, the usual and most notable connections in Finsler geometry~\cite{Bao} (the Berwald, Cartan, Chern-Rund and Hashiguchi connections), which are homogeneous but nonlinear, can be properly and conveniently be defined as linear connections on the pullback of the tangent bundle, with domain $\slit{τ}_M^*(TM)$.

\section{Natural prolongation}
\label{section:natural.prolongation}

The procedure described in this paper assigning a linear connection to a non linear one can be properly understood as a first order prolongation in the context of natural operations on fibered manifolds~\cite{NatOp}. 

A systematic study of the functorial prolongation of a connection on a (non necessarily linear) fiber bundle $B\to M$ to a connection on the vertical bundle has been carried out in~\cite{DouMi2, DouMi1, DouMi3}. In~\cite{DouMi1} it is proved that a bundle functor $G$ on the category of fibred manifolds admits a functorial operator lifting connections on $B\to M$ to connections on $GB\to B$ if and only if the functor $G$ is a trivial bundle functor $B\mapsto GB=B\times W$, for some manifold $W$. As a consequence of this result, since the vertical functor is not a trivial functor, we deduce that there is no natural operator transforming connections on $B\to M$ to connections on $\Ver{B}\to B$. However, as it is already remarked in~\cite{DouMi1}, the obstruction for the existence of such a natural operator may disappear if we consider some additional structure. In~\cite[\S 46.10]{NatOp}, by restricting to the subcategory of affine bundles, it is shown that there exists a 1-parameter family of first order natural operators (natural on local isomorphisms of affine bundles) transforming connections on $\map{π}{E}{M}$ into connections on $\map{τ_E^\V\equivτ_E|_{\Ver{E}}}{\Ver{E}}{E}$. 

The pullback bundle $\map{π^*π}{π^*E}{E}$ can be identified (via the vertical lifting $\vlift_π$) with the vertical bundle $\map{τ_E^\V}{\Ver{E}}{E}$. Under this identification the linearized connection can be considered as a connection on the vertical bundle. 

Our linearization procedure transforms a connection on a vector bundle (a particular case of an affine bundle) into a linear connection on the pullback bundle, or equivalently to the vertical bundle. From the definition (Theorem~\ref{linear.connection.0}) it is clear that our construction is a natural first order differential operator. Therefore it should coincide with one of the members of the above mentioned family.

In~\cite{linearization.1} we provided an intrinsic expression for the members of that family. This expression was found by `brute force' from the coordinate expression given in~\cite{NatOp}, but no geometrical understanding of the situation was provided. In this section we will find in more clear intrinsic terms the above family and we will prove that these are all natural first order prolongations of the given connection. We will try to keep it as elementary as possible, and we refer to the reader to~\cite{NatOp} for further details of the theory of natural operators.

\medskip

The functor that we are considering here is the pullback functor $\operatorname{Pb}$ from the category of vector bundles to the  category of vector bundles. The functor $\operatorname{Pb}$ assigns the vector bundle $\operatorname{Pb}(E)=π^*E$ to a vector bundle $\map{π}{E}{M}$, and the morphism $\operatorname{Pb}(ϕ)=(ϕ,ϕ)$ to a vector bundle morphism $\map{ϕ}{E}{E'}$ between vector bundles.

For every real number $λ∈ℝ$ we can define the family of maps $B_λ$ by 
\[
B_λ(a,b)w = B(a,b)w+λ\vlift_π(b,κ(w))
\]
for $(a,b)∈π^*E$ and $w∈T_aE$. It is easy to see that $\bigl((a,b),w\bigr) \mapsto \bigl(w,B_λ(a,b)w\bigr)$ is an Ehresmann connection on $π^*π$, which will be denoted symbolically by $\ol{\Hor}+λκ^\V$. 

\begin{theorem}
All first order operators transforming a connection on a vector bundle into a connection on the pullback bundle which are natural on the local isomorphisms of vector bundles form the 1-parameter family $\ol{\Hor}+λκ^\V$, for $λ∈ℝ$.
\end{theorem}

\begin{proof}
The difference between two connections on $π^*E$ is of the form 
\[
\hlift_2\bigl((a,b),w\bigr)-\hlift_1\bigl((a,b),w\bigr) 
= \bigl(w, B_2(a,b)w\bigr)-\bigl(w, B_1(a,b)w\bigr)
= \bigl(0,Δ(a,b)w\bigr)
\]
for $Δ(a,b)$ a linear map $\map{Δ(a,b)}{T_aE}{T_bE}$ with values in the vertical bundle. We can write $Δ$ in the form $Δ(a,b)w=\vlift_π(b,D(a,b)w)$ where $D(a,b)$ is a linear map $\map{D(a,b)}{T_aE}{E_{π(a)}}$.

Taking $\hlift_2$ an arbitrary connection on $π^*E$ and $\hlift_1$ the linearized connection of $\Hor$, the associated maps $B'$ and $B$ are related by $B'(a,b)w=B(a,b)w+\vlift_π\bigl(b, D(a,b)w\bigr)$ for linear maps $\map{D(a,b)}{T_aE}{E_{π(a)}}$ depending smoothly on the point $(a,b)∈π^*E$. Taking the horizontal and vertical components of $w$ we can write $w=\hlift(a,v)+\vlift_π(a,c)$ with $v=Tπ(w)$ and $c=κ(w)$. Thus we have $D(a,b)w=D_\H(a,b)v+D_{\V}(a,b)c$, where $D_\H(a,b)v=D(a,b)\hlift(a,v)$ and $D_{\V}(a,b)c=D(a,b)\vlift_π(a,c)$. 

It follows that $B'$ is a natural operator if and only if both $D_{\H}$ and $D_{\V}$ are natural transformations, which means that for every local diffeomorphism $(ϕ,φ)$ of $E$ they satisfy
\begin{align*}
ϕ\bigl(D_{\H}(a,b)v\bigr)&=D_{\H}\bigl(ϕ(a),ϕ(b)\bigr)Tφ(v),\\
ϕ\bigl(D_{\V}(a,b)c\bigr)&=D_{\V}\bigl(ϕ(a),ϕ(b)\bigr)ϕ(c).
\end{align*}
We now prove that these conditions imply that $D_{\H}=0$ and that $D_\V$ is a constant multiple of $(a,b)\mapsto \id_{T_{π(a)}M}$.

For that we just need to consider the trivial bundle $\map{\pr_1}{E=B\times F}{B}$, where $B$ and $F$ are Banach spaces. An element $((x,y),(x,z))$ of $π^*E$ will be written $(x,y,z)∈B\times F\times F$. The map $D_\H$ will be of the form $D_\H(x,y,z)(x,v)=\bigl(x,\bar{D}_\H(x,y,z)v\bigr)$ with $\bar{D}_\H(x,y,z)$ a linear map from $B$ to $F$ depending smoothly on $(x,y,z)$. For every $k∈ℝ$, the invariance under the morphism $ϕ(x,y)=(kx,y)$ over the map $φ(x)=kx$ reads $\bar{D}_{\H}(kx,y,z)(kv)=\bar{D}_{\H}(x,y,z)v$, and hence it implies $\bar{D}_{\H}$ is homogeneous of degree $-1$ with respect to the variable $x$. Smoothness of $\bar{D}_\H$ implies that it must vanish, hence $D_\H=0$.

The map $D_\V$ will be of the form $D_\V(x,y,z)(x,u)=\bigl(x,\bar{D}_\V(x,y,z)u\bigr)$ with $\bar{D}_\V(x,y,z)$ a linear map from $F$ to $F$ depending smoothly on $(x,y,z)$. For $k∈ℝ$, we take the morphism $ϕ(x,y)=(x,ky)$ over the identity $ϕ(x)=x$. The invariance condition for $D_{\V}$ reads $k\bar{D}_\V(x,y,z)u=\bar{D}_{\V}(x,ky,kz)ku$, or in other words $\bar{D}_\V(x,y,z)u=\bar{D}_{\V}(x,ky,kz)u$. Thus $\bar{D}_\V$ is homogeneous of degree 0 in the variables $(y,z)$. Smoothness of $\bar{D}_\V$ implies that it is independent on $(y,z)$, and hence it depends only on the base point, i.e. $D_\V(a,b)\equiv D_\V(m)$ with $m=π(a)=π(b)$. 

For any linear automorphism $A$ of $F$ we consider the diffeomorphism $ϕ(x,y)=(x,Ay)$. The invariance condition for $D_\V$ reads $A\bar{D}_\V(x)u=\bar{D}_\V(x)Au$. Therefore the endomorphism $\bar{D}_\V(x)$ commutes with any regular endomorphism, from where it follows that it is a multiple the identity. Thus $D_\V(a,b)=α(m)\id_{E_m}$ for some smooth function $α∈\cinfty{M}$. Finally taking again the morphism $ϕ(x,y)=(kx,y)$ we get $α(kx)u=α(x)u$, so that $α$ is homogeneous of degree 0 and smooth, and hence constant, $α(x)=λ$.

We conclude that $D(a,b)w=λκ(w)$ for some constant $λ∈ℝ$, and therefore $B'(a,b)w=B(a,b)w+\vlift_π(b,λκ(w))$, for all $(a,b)∈π^*E$ and all $w∈T_aE$.
\end{proof}

For $λ=0$ we obtain the connection $\ol{\Hor}$, which is linear. For $λ≠0$ the connection $\ol{\Hor}+λκ^\V$ is not a linear connection but an affine connection on the bundle $π^*E$, with its canonical affine structure, whose associated linear connection is precisely $\ol{\Hor}$. Indeed, 
\begin{align*}
B_λ(a,b+b')w
&=B(a,b+b')w+\vlift_π(b+b',λκ(w))\\
&=\Bigl(B(a,b)w\Tplus B(a,b')w\Bigr)+\Bigl(\frac{d}{ds}(b+b'+sλκ(w))\at{s=0}\Bigr)\\
&=\Bigl(B(a,b)w\Tplus B(a,b')w\Bigr)+\Bigl(\vlift_π(b,λκ(w))\Tplus 0_{b'}\Bigr)\\
&=\Bigl(B(a,b)w+\vlift_π(b,λκ(w))\Bigr)\Tplus\Bigl(B(a,b')w+0_{b'}\Bigr)\\
&=B_λ(a,b)w \Tplus B(a,b')w,
\end{align*}
where in the first step we have used linearity of $B(a,b)$ in the argument $b$, and in the third one we have used the interchange law (see Section~\ref{section:preliminaries}).

Also, the linearized connection $\ol{\Hor}$ is semibasic, while the other members in the family are not. Indeed, if $w∈T_aE$ is vertical then $B(a,b)w=0$, and if we write $w=\vlift_π(a,c)$ for some $c∈E$, then
\[
B_λ(a,b)w=B(a,b)w+λ\vlift_π(b,κ(w))=0+λ\vlift_π(b,c)=λ\vlift_π(b,c),
\]
which does not identically vanishes, except for $λ=0$.

We have proved the following statement.
\begin{proposition}
\label{natural+=linearization}
The linearization $\ol{\Hor}$ of a connection $\Hor$ can be characterized by any of the following properties:
\begin{itemize}
\item $\ol{\Hor}$ is the only natural first order prolongation of $\Hor$ which is linear.
\item $\ol{\Hor}$ is the only natural first order prolongation of $\Hor$ which is semibasic.
\end{itemize}
\end{proposition}

\medskip

In the finite dimensional case, in a local coordinate system the expression of the maps $B$ for the connection $\ol{\Hor}+λκ^\V$ is
\begin{align*}
B_λ(a,b)
\Bigl(w^i\pd{}{x^i}&+w^A\pd{}{y^A}\Bigr)\at{(x^i,y^A)}
=\\
{}={}&\left(w^i\pd{}{x^i}+\Bigl(λ\bigl[w^A-Γ^A_i(x,y)w^i\bigr]-Γ^A_{iB}(x,y)z^Bw^i\Bigr)\pd{}{y^A}\right)\at{(x^i,z^A)}.
\end{align*}
It is clear that this connection is affine (the connection coefficients are affine functions of $z^A$) and is linear only in the case $λ=0$, which is the case of the linearization. The annihilator of the horizontal distribution is spanned by the 1-forms
\[
dz^A+Γ^A_{iB}(x,y)z^Bdx^i+λ\bigl(dy^A-Γ^A_i(x,y)dx^i\bigr).
\]

\section{The covariant derivative}
\label{section:covariant.derivative}

In this section we will find an explicit expression of the covariant derivative associated to the linearization of an Ehresmann connection in terms of brackets of vector fields.

As a preliminary result we have the following expression for the associated Dombrowski connection map. 

\begin{proposition}
The connection map $\map{\bar{κ}}{T(π^*E)}{π^*E}$ for the linearized connection $\ol{\Hor}$ is given by 
\[
\bar{κ}(w,w')=\bigl(a,ν_π\bigl(w'-B(a,b)w\bigr)\bigr),
\]
for $(w,w')∈T_{(a,b)}(π^*E)$.
\end{proposition}
\begin{proof}
The vertical lift on the bundle $\map{π^*π}{π^*E}{E}$ is given by
\[
\vlift_{π^*π}\bigl((a,b),(a,c)\bigr)=\bigl(0_a,\vlift_π(b,c)\bigr).
\]
Indeed, 
\[
\vlift_{π^*π}\bigl((a,b),(a,c)\bigr)
=\frac{d}{ds}\bigl((a,b)+s(a,c)\bigr)\at{s=0}
=\frac{d}{ds}(a,b+sc)\bigr)\at{s=0}
=\bigl(0_a,\vlift(b,c)\bigr).
\]
Therefore the vertical projection is $ν_{π^*π}(0_a,w)=(a,ν_π(w))$ for $w∈T_bE$ vertical.

The vertical projector $P_{\bar{\V}}$ for the linearized connection $\ol{\Hor}$ is given by
\[
P_{\bar{\V}}(w,w')
=(w,w')-ξ^{\bar{\H}}\bigl((a,b),w\bigr)
=(w,w')-(w,B(a,b)w)
=(0_a,w'-B(a,b)w).
\]
Hence
\[
\bar{κ}(w,w')=ν_{π^*π}(P_\V(w,w'))=ν_{π^*π}(0_a,w'-B(a,b)w)=\bigl(a,ν_π(w'-B(a,b)w)\bigr),
\]
which proofs the statement.
\end{proof}

A section $ζ∈\sec{π^*π}$ is of the form $ζ(a)=(a,\ul{ζ}(a))$ where $\map{\ul{ζ}}{E}{E}$ is the corresponding section along $π$. From the result above we have that the covariant derivative of the section $ζ$ in the direction of a vector $w∈T_aE$ is
\[
D_wζ=\bar{κ}(T_aζ(w))=\bar{κ}\bigl(w,T_a\ul{ζ}(w)\bigr)=\bigl(a,ν_π\bigl(T_a\ul{ζ}(w)-B(a,\ul{ζ}(a))w\bigr)\bigr),
\]
where $m=π(a)$. 

From this expression it is clear that the formula for the covariant derivative is simpler when expressed in terms of sections along $π$. Therefore we will consider the covariant derivative as a map $\map{D}{\vectorfields{E}\times\sec[π]{π}}{\sec[π]{π}}$ which is given by
\[
D_wσ=ν_π\bigl(T_aσ(w)-B\bigl(a,σ(a)\bigr)w\bigr),\qquad σ∈\sec[π]{π},\ w∈T_aE.
\]

As a consequence of the above formula we have the following result.

\begin{theorem}
\label{covariant.derivative}
The covariant derivative associated to the linearized connection $\overline{\Hor}$ is given by 
\[
D_Wσ=κ\bigl([P_\H(W),σ^\V]+Tσ(P_\V(W))\bigr)
\]
for every section $σ$ of $E$ along $π$ and every vector field $W∈\vectorfields{E}$.
\end{theorem}

\begin{proof}
We recall that on any manifold $N$ the bracket $[X,Y]$ of two vector fields $X,Y∈\vectorfields{N}$ can be written in the form (see~\cite[\S \textbf{6.14}]{NatOp}) 
\[
ν_{τ_N}∘(TY∘X-\caninv{TN}∘TX∘Y)=[X,Y],
\]
or equivalently
\[
\vlift_{τ_N}(Y,[X,Y])=TY∘X-\caninv{TN}∘TX∘Y.
\]

\smallskip

In our concrete case the manifold $N$ is $N=E$. For a basic vector field $X∈\vectorfields{M}$ and a section $σ$ along $π$ we have 
\[
\tag{$*$}
\vlift_{τ_E}(σ^\V,[X^\H,σ^\V])=Tσ^\V○X^\H-\caninv{TE}○TX^\H○σ^\V.
\]
Taking into account that $ν_π○σ^\V=σ$ and $\hlift(a+sb,X(m))=X^\H(a+sb)$, and using $ν_{Tπ}=Tν_π○\caninv{TE}$ (Proposition~\ref{involution.vertical} \thetag{2}), we have
\begin{align*}
Tσ(X^\H(a))&-B(a,σ(a))X^\H(a)=\\
&=Tσ(X^\H(a))-Tν_π○\caninv{TE}\Bigl(\frac{d}{ds}X^\H(a+sσ(a))\at{s=0}\Bigr)\\
&=T(ν_π○σ^\V)(X^\H(a))-Tν_π○\caninv{TE}○TX^\H(σ^\V(a))\\
&=Tν_π\Bigl(Tσ^\V(X^\H(a))-\caninv{TE}○TX^\H(σ^\V(a))\Bigr)\\
&=Tν_π○\vlift_{τ_E}\bigl(σ^\V(a),[X^\H,σ^\V](a)\bigr)\\
&=\vlift_π\bigl(σ(a),ν_π([X^\H,σ^\V](a))\bigr),
\end{align*}
where in the last step we have used Proposition~\ref{several.relations} \thetag{3}, and we have taken into account that $[X^\H,σ^\V]$ is vertical because $X$ is basic. Therefore
\[
(D_{X^\H}σ)(a)=ν_π\bigl([X^\H,σ^\V](a)\bigr).
\]

For a vertical vector $w∈\ver[a]{π}$ we have $B(a,b)w=0$, because $T_aπ(w)=0$. Hence 
\[
D_vσ=ν_π\bigl(Tσ(v)-B(a,σ(a))v\bigr)=ν_π\bigl(Tσ(v)\bigr).
\]

Up to now we have proved that
\[
D_{X^\H+η^\V}\,σ=ν_π\bigl([X^\H,σ^\V]+Tσ(η^\V)\bigr),
\]
for $η∈\sec{π^*E}$ and $X∈\vectorfields{M}$. In other words
\[
D_Yσ=ν_π\bigl([P_\H(Y),σ^\V]+Tσ(P_\V(Y))\bigr),
\]
for any projectable vector field $Y∈\vectorfields{E}$.

For a general vector field $W∈\vectorfields{E}$ the bracket $[P_\H(W),σ^\V]$ is no longer vertical, and we cannot apply $ν_π$ directly. Projecting first to the vertical bundle and then applying $ν_π$ we have that the expression
\[
ν_π\bigl(P_\V([P_\H(W),σ^\V])\bigr)=κ([P_\H(W),σ^\V])
\]
is $\cinfty{TE}$-linear in $W$, from where the result follows.

Alternatively, one can check explicitly (as we did in~\cite{linearization.1}) that the expression on the right hand side of the formula in the statement satisfies the properties of a linear connection on $π^*E$. As every vector in $TE$ can be obtained as the value of a projectable vector field the result follows.
\end{proof}

Notice that a section $σ$ along $π$ is basic if and only if $D_wσ=0$ for all vertical vector~$w$. 

The expression for the covariant derivative in Theorem~\ref{covariant.derivative} was the starting point in~\cite{linearization.1}. In such paper the reader can find further information about several applications of the theory.

\section{Fibre derivative, curvature, and parallel transport}
\label{section:fibre.derivative}

It is clear from the formulas above that procedure of linearization of a connection has to do with the fiber derivative. In this section we will make precise this relationship.

Given a (generally non linear) bundle map $\map{ϕ}{E}{F}$ between two vector bundles $\map{π}{E}{M}$ and $\map{τ}{F}{N}$ fibered over a map $\map{φ}{M}{N}$ the fibre derivative of~$ϕ$ is the vector bundle map $\map{\calfϕ}{π^*E}{τ^*F}$ fibered over $ϕ$ determined by 
\[
\vlift_τ∘\calfϕ=Tϕ∘\vlift_π.
\]
In more explicit terms, for $(a,b)∈π^*E$ the fibre derivative of $ϕ$ is of the form $\calfϕ(a,b)=(ϕ(a),\calf_aϕ(b))$ where $\calf_aϕ$ is the restriction of $\calf ϕ$ to the fibre over $a$ and is given by $\calf_aϕ(b)=ν_τ∘Tϕ\bigl(\vlift_π(a,b)\bigr)$.

For the identity map in $E$ we have $\calf\id_E=\id_{π^*E}$. If $\map{ρ}{G}{P}$ is a third vector bundle and $\map{ψ}{F}{G}$ is another fibered map then $\calf{(ψ∘ϕ)}=\calfψ∘\calfϕ$. These properties follow easily from the definition.

\medskip

We next show that the fibre derivative of a projectable vector field on a vector bundle defines a linear vector field on the pullback bundle. We recall~\cite{Mackenzie} that a vector field $Z∈\vectorfields{F}$ on a vector bundle $\map{τ}{F}{N}$ is said to be linear if it satisfies one of the following two equivalent properties 
\begin{enumerate}
\item The flow of $Z$ is by vector bundle automorphisms.
\item $Z$ is a vector bundle morphism from $\map{τ}{F}{N}$ to $\map{Tτ}{TF}{TN}$.
\end{enumerate}
There is a one to one correspondence between linear vector fields on $\map{τ}{F}{N}$ and derivations of the $\cinfty{M}$-module $\sec{F}$. If $\{ϕ_s\}$ is the local flow of a linear vector field $Z$ then the derivation $\cald$ associated to $Z$ satisfies
\[
\cald η(p)=\frac{d}{ds}\bigl(ϕ_s^*η\bigr)(p)\at{s=0}=\lim_{s→0}\frac{ϕ_{-s}\bigl(η(φ_s(p))-η(p)\bigr)}{s}
\]
where $\{φ_s\}$ is the flow on the base (the local flow of the vector field in $N$ to which $Z$ projects). See~\cite{Mackenzie} for the details. 

\medskip

We recall that if $U∈\vectorfields{P}$ is a vector field on a manifold $P$ with local flow $\{φ_s\}$ then the infinitesimal generator of the flow $\{Tφ_s\}$ is the complete lift $U^\C∈\vectorfields{TP}$ of $Z$. It is related to the tangent of $U$ by the canonical involution $\caninv{TP}∘TU=U^\C$.

\begin{proposition}
\label{linearization.of.a.vector.field}
Let $Y∈\vectorfields{E}$ be a projectable vector field. 
\begin{enumerate}
\item The fibre derivative $\calf Y$ is a linear vector field on $π^*E$. 
\item If $\{ϕ_s\}$ is the local flow of $Y$ then $\{\calf ϕ_s\}$ is the local flow of $\calf Y$. 
\item The derivation $\map{\cald^Y}{\sec[π]{E}}{\sec[π]{E}}$ associated to the linear vector field $\calf Y$ is $\cald^Yσ=ν_π\bigl([Y,σ^\V]\bigr)$.
\end{enumerate}
If $Y'∈\vectorfields{E}$ is another projectable vector field then
\begin{enumerate}
\setcounter{enumi}{3}
\item $[\calf Y,\calf Y']=\calf [Y,Y']$, and
\item $[\cald^Y,\cald^{Y'}]=\cald^{[Y,Y']}$.
\end{enumerate}
\end{proposition}

\begin{proof}
From the definition of the fibre derivative we have that $\calf Y$ is a map  $\map{\calf Y}{π^*E}{(Tπ)^*(TE)}$ projecting onto $Y$, i.e. $\pr_1∘\calf Y=Y∘(π^*π)$. Moreover,\leavevmode
\begin{align*}
(τ_E,τ_E)(\calf Y(a,b))
&=\bigl(τ_E(Y(a)),τ_E\bigl(\calf_aY(b)\bigl)\bigr)\\
&=\bigl(a,τ_E∘ν_{Tπ}∘TY(\vlift_π(a,b))\bigr)\\
&=\bigl(a,ν_π∘Tτ_E∘TY(\vlift_π(a,b))\bigr)\\
&=\bigl(a,ν_π(\vlift_π(a,b))\bigr)\\
&=(a,b),
\end{align*}
where we have used Proposition~\ref{several.relations} \thetag{1}.

\begin{wrapfigure}{l}{0.36\textwidth}
$\xymatrix{
&π^*E\ar[d]_{π^*π}\ar[r]^{\calf Y}&T(π^*E)\ar[d]^{T(π^*π)}\\
&E\ar[d]_{π}\ar[r]^{Y}&TE\ar[d]^{Tπ}\\
&M\ar[r]_{X}&TM
}\qquad$
\end{wrapfigure}

Under the canonical isomorphism  $T(π^*E)\simeq (Tπ)^*(TE)$ we have that $(τ_E,τ_E)$ corresponds to $τ_{π^*E}$, so that we have got a map $\map{\calf Y}{π^*E}{T(π^*E)}$ such that  $τ_{π^*E}∘\calf Y=\id_{π^*E}$. Moreover, the fibre derivative of a map is a vector bundle map, and thus we conclude that $\calf Y$ is a linear vector field on $π^*E$. This proves \thetag{1}.

We now prove \thetag{2}. From the definition of the fibre derivative we have $\vlift_{Tπ}∘\calf Y=TY∘\vlift_π$. Applying the canonical involution to both sides we get 
\[
T\vlift_π∘\calf Y=Y^\C∘\vlift_π,
\]
where we have used $\caninv{TE}∘\vlift_{Tπ}=T\vlift_π$ (Proposition~\ref{involution.vertical} \thetag{1}) and $\caninv{TE}∘TY=Y^\C$. Thus $\calf Y$ is $\vlift_π$-related to $Y^\C$ from where it follows that the local flow  $\{Φ_s\}$ of $\calf Y$ satisfies $\vlift_π∘Φ_s=Tϕ_s∘\vlift_π$. This is the defining equation for the fibre derivative of $ϕ_s$, so that we conclude $Φ_s=\calf ϕ_s$.

For identifying the derivation associated to $\calf Y$, for $a∈E$ we have
\begin{align*}
\vlift_π\bigl((\calfϕ_s^*\,σ)(a)\bigr)
&=\vlift_π\bigl(\calfϕ_{-s}(σ(ϕ_s(a)))\bigr)\\
&=Tϕ_{-s}\bigl(σ^\V(ϕ_s(a))\bigr).
\end{align*}
Computing $\vlift_π(\cald^Yσ(a))$ by means of the limit
\begin{align*}
\vlift_π(\cald^Yσ(a))
&=\vlift_π\left(\lim_{s→0}\frac{1}{s}[(\calfϕ_s^*\,σ)(a)-σ(a)]\right)\\
&=\lim_{s→0}\frac{1}{s}[Tϕ_{-s}\bigl(σ^\V(ϕ_s(a))\bigr)-σ^\V(a)]\\
&=\call_Yσ^\V(a),
\end{align*}
from where $\cald^Yσ(a)=ν_π\bigl(\call_Yσ^\V(a)\bigr)$. This proves \thetag{3}.

\thetag{4} and \thetag{5} are clearly equivalent, so that we will just prove \thetag{4}. If $\{ϕ_s\}$ is the flow of $Y$ and $\{ϕ'_s\}$ is the local flow of $Y'$ then the commutator
$c(s)=ϕ_s∘ϕ'_s∘ϕ_{-s}∘ϕ'_{-s}$ satisfies $c(0)=\id_E$, $c'(0)=0$ and $2·c''(0)=[Y,Y']$. Applying the fibre derivative we get $\calf(ϕ_s∘ϕ'_s∘ϕ_{-s}∘ϕ'_{-s})=\calf ϕ_s∘\calf ϕ'_s∘\calf ϕ_{-s}∘\calf ϕ'_{-s}$. Taking the second derivative with respect to $s$ at $s=0$ on the left hand side we get $\calf[Y,Y']$ and on the right hand side we get $[\calf Y,\calf Y']$.  
\end{proof}

For a projectable vector field $Y$ the derivations $\cald^Y$ and $D_Y$ are closely related but they are not equal.

\begin{proposition}
\label{derivation.vs.covariant.derivative}
If $Y∈\vectorfields{E}$ is a projectable vector field then
\[
\cald^Yσ=D_Yσ-D_{σ^\V}(κ(Y)),
\]
for every section $σ$ along $π$.
\end{proposition}
\begin{proof}
If $Y∈\vectorfields{E}$ is projectable over $X∈\vectorfields{M}$ we can write $Y=X^\H+η^\V$ where $η=κ(Y)∈\sec[π]{E}$. Then 
\[
\cald^Yσ=ν_π\bigl([X^\H,σ^\V]+[η^\V,σ^\V]\bigr),
\]
while
\[
D_Yσ=ν_π\bigl([X^\H,σ^\V]+Tσ(η^\V)\bigr).
\]
Notice that $D_{η^\V}σ=ν_π(Tσ(η^\V))$, from where it is easy to see that $ν_π([η^\V,σ^\V])=D_{η^\V}σ-D_{σ^\V}η$. Therefore the difference between both derivations is
\[
D_Yσ-\cald^Yσ=ν_π(Tσ(η^\V))-ν_π\bigl([η^\V,σ^\V]\bigr)=ν_π\bigl(Tη(σ^\V)\bigr)=D_{σ^\V}η,
\]
which proves the statement.
\end{proof}

It is apparent that the following class of vector fields will be very important in the description and analysis of the linearized connection.

\begin{definition}
A vector field $Y∈\vectorfields{E}$ is said to be $\Hor$-basic if it is projectable and $κ(Y)$ is a basic section.
\end{definition}

In other words, a vector field $Y∈\vectorfields{E}$ is $\Hor$-basic if it can be written as $Y=X^\H+η^\V$ with $X∈\vectorfields{M}$ (a basic vector field) and $η∈\sec{E}$ (a basic section). The set of $\Hor$-basic vector fields is a $\cinfty{M}$-module. From Proposition~\ref{derivation.vs.covariant.derivative} we have that a vector field $Y$ is $\Hor$-basic if and only if $\cald^Y=D_Y$.

\smallskip

For a $\Hor$-basic vector field $Y$ on $E$ the fibre derivative $\calf Y$ of $Y$ equals the horizontal lift of $Y$ with respect to the linearized connection.
    
\begin{theorem}
\label{horizontal=fibre.derivative}
The horizontal lift of a $\Hor$-basic vector field $Y$ with respect to the linearized connection $\ol{\Hor}$ is equal to the fibre derivative of $Y$, that is $Y^{\bar{\H}}=\calf Y$. 
\end{theorem}

\begin{proof}
For a horizontal projectable vector field $Y=X^\H$, with $X∈\vectorfields{M}$, we have
\begin{align*}
(\calf X^\H)(a,b)
&=\bigl(X^\H(a),\calf_aX^\H(b)\bigr)\\
&=\bigl(X^\H(a),ν_{Tπ}(TX^\H(\vlift_π(a,b)))\bigr)\\
&=\left(X^\H(a),ν_{Tπ}\Bigl(\frac{d}{ds}X^\H(a+sb)\at{s=0}\Bigr)\right)\\
&=\left(X^\H(a),ν_{Tπ}\Bigl(\frac{d}{ds}\hlift(a+sb,X(m))\at{s=0}\Bigr)\right)\\
&=\bigl(X^\H(a),B(a,b)X^\H(a)\bigr)\\
&=ξ^{\bar{\H}}\bigl((a,b),X^\H(a)\bigr).
\end{align*}

For a vertical vector field $Y=η^\V$, with $η∈\sec{E}$ a basic section, using that $ν_{Tπ}=Tν_π∘\caninv{TE}$ on $Tπ$-verticals (Proposition~\ref{involution.vertical} \thetag{2}), we have 
\begin{align*}
\calf_aη^\V(b)
&=ν_{Tπ}(Tη^\V(\vlift_π(a,b)))\\
&=ν_{Tπ}\Bigl(\frac{d}{ds}η^\V(a+sb)\at{s=0}\Bigr)\\
&=ν_{Tπ}\Bigl(\frac{d}{ds}\vlift_π(a+sb,η(m))\at{s=0}\Bigr)\\
&=Tν_π∘\caninv{TE}\Bigl(\frac{d}{ds}\frac{d}{dt}\bigl(a+sb+tη(m)\bigr)\at{t=0}\at{s=0}\Bigr)\\
&=Tν_π\Bigl(\frac{d}{dt}\frac{d}{ds}\bigl(a+sb+tη(m)\bigr)\at{s=0}\at{t=0}\Bigr)\\
&=Tν_π\Bigl(\frac{d}{dt}\vlift_π\bigl(a+tη(m),b\bigr)\at{t=0}\Bigr)\\
&=\frac{d}{dt}ν_π(\vlift_π(a+tη(m),b))\at{t=0}\\
&=\frac{d}{dt}\,b\;\at{t=0}\\
&=0_b,
\end{align*}
so that 
\[
(\calf η^\V)(a,b)
=\bigl(η^\V(a),\calf_aη^\V(b)\bigr)
=(η^\V(a),0_b)
=ξ^{\bar{\H}}\bigl((a,b),η^\V(a)\bigr).
\]

For a general $\Hor$-basic vector field $Y$ we write $Y=X^\H+η^\V$ with both $X$ and $η$ basic, and hence by linearity of the horizontal lift and the fibre derivative we have $Y^{\bar{\H}}=(X^\H)^{\bar{\H}}+(η^\V)^{\bar{\H}}=\calf (X^\H)+\calf (η^\V)=\calf(X^\H+η^\V)=\calf Y$.
\end{proof}

\paragraph{Parallel transport}
As an immediate consequence of Proposition~\ref{linearization.of.a.vector.field}, if the local flow of a $\Hor$-basic vector field $Y$ is $ϕ_s$ then the local flow of $Y^{\bar{\H}}$ is $\calfϕ_s$. Thus for $\Hor$-basic vector fields we have an explicit expression of the parallel transport system defined by the linear connection $\ol{\Hor}$, which complements the description given in~\cite{linearization.1}.

\begin{theorem}
\label{parallel.transport}
Let $Y∈\vectorfields{E}$ be a $\Hor$-basic vector field with local flow $\{ϕ_s\}$. For the linearized connection $\ol{\Hor}$ the parallel transport map along a flow line $s\mapsto ϕ_s(a)$ of $Y$  is $\calfϕ_s$. 

Conversely, if $\Hor$ is a connection on $E$ and $\Hor'$ is a connection on $π^*E$ satisfying the above property then $\Hor'$ is the linearization of $\Hor$.
\end{theorem}
\begin{proof}
The parallel transport of a vector $(a,b)∈π^*E$ along $γ(s)=ϕ_s(a)$ is given by the horizontal curve (with respect to $\ol{\Hor}$) projecting onto the curve $γ$ with initial value $(a,b)$, which is but the integral curve of $Y^{\bar{\H}}$ starting at $(a,b)$. Since $Y^{\bar{\H}}=\calf Y$, such horizontal curve is but the integral curve of $\calf Y$ starting at $(a,b)$, which is $\calf ϕ_s(a,b)$.

Conversely, assume $\Hor'$ is a connection on $π^*E$ satisfying that the horizontal lifting of the curve $s\mapsto ϕ_s(a)$ starting at $(a,b)$ is $s\mapsto \calf ϕ_s(a,b)$ for every flow line $ϕ_s$ of every $\Hor$-basic vector field $Y$ and every $(a,b)∈π^*E$. The tangent vector to the curve $\calf ϕ_s(a,b)$ at $s=0$ is hence a horizontal vector 
\[
\frac{d}{ds}\calf ϕ_s(a,b)\at{s=0}
=\calf Y(a,b)∈\Hor'_{(a,b)}.
\]
On the other hand, from Theorem~\ref{horizontal=fibre.derivative} we know that $\calf Y(a,b)∈\ol{\Hor}_{(a,b)}$. As every tangent vector to $E$ at the point $a$ can be obtained by evaluating a $\Hor$-basic vector field at $a$ it follows that $\Hor'_{(a,b)}=\ol{\Hor}_{(a,b)}$. This being true for every $(a,b)∈π^*E$ implies that the connection $\Hor'$ and $\ol{\Hor}$ are equal.
\end{proof}

For the vertical lifting $Y=η^\V$ of a basic section $η∈\sec{E}$ we can provide an explicit formula for the flow of $\calf Y$. The flow of $η^\V$ is $ϕ_s(a)=a+sη(m)$, with $m=π(a)$. Thus
\[
Tϕ_s\bigl(\vlift_π(a,b)\bigr)=\frac{d}{dt}ϕ_s(a+tb)\at{t=0}=\frac{d}{dt}(a+sη(m)+tb)\at{t=0}=\vlift_π\bigl(a+sη(m),b\bigr),
\]
from where we get $\calfϕ_s(a,b)=(a+sη(m),b)$. 

In other words, along any straight line contained in a fibre $E_m$ connecting two points $a$ and $a'$ the parallel transport map is $(a,b)\mapsto (a',b)$. This result holds true for any vertical curve (i.e. a curve contained in a fibre).

\begin{proposition}
\label{complete.parallelism}
For the linearized connection, parallel transport along a vertical curve is the standard parallelism on the fibre (as a vector space).
\end{proposition}
\begin{proof}
In fact we will prove that the result holds true for any semibasic connection $\Hor'$ in $π^*E$. Let $\map{γ}{[t_0,t_1]}{E}$ be a vertical curve connecting the points $a=γ(t_0)$ and $a'=γ(t_1)$. A vertical curve $γ(t)$ is entirely contained in a fibre, $γ(t)∈E_{π(a)}$. The horizontal lift $γ^{\H'}$ of the curve $γ$ with initial value $(a,b)$ is $γ^{\H'}(t)=(γ(t),b)$. Indeed, the curve $γ^{\H'}$ is horizontal  
\[
\frac{dγ^{\H'}}{dt}(t)=(\dot{γ}(t),0_b)=ξ^{\H'}\bigl((γ(t),b),\dot{γ}(t)\bigr)∈\Hor'_{γ(t)},
\]
because the connection is semibasic and $\dot{γ}(t)$ is a vertical vector, and obviously $γ^{\H'}(t_0)=(a,b)$. Therefore, the parallel translation of the vector $(a,b)=γ^{\H'}(t_0)$ is the vector $γ^{\H'}(t_1)=(a',b)$.
\end{proof}

\paragraph{Curvature}
We will find the curvature of the connection $\ol{\Hor}$. We denote by $R$ the curvature 2-form of the nonlinear connection $\Hor$. It is defined by 
\[
R(X,Y)=κ\bigl([X^\H,Y^\H]\bigr)=ν_π\bigl([X^\H,Y^\H]-[X,Y]^\H\bigr)
\]
for $X,Y∈\vectorfields{M}$.

\begin{theorem}
\label{curvature}
The curvature of the linear connection $\ol{\Hor}$ can be written as
\[
\curv(Y_1,Y_2)σ=-D_{σ^\V}\Bigl(R(X_1,X_2)+D_{X_1^\H}η_2-D_{X_2^\H}η_1\Bigr)
\]
for $\Hor$-basic vector fields $Y_i=X_i^\H+η_i^\V$, $i=1,2$, and any $σ∈\sec[π]{E}$. 
\end{theorem}
\begin{proof}
For $Y_1=X_1^\H+η_1^\V$, $Y_2=X_2^\H+η_2^\V$ two $\Hor$-basic vector fields we have
\begin{align*}
κ([Y_1,Y_2])
&=κ([X_1^\H+η_1^\V, X_2^\H+η_2^\V])\\
&=κ\bigl([X_1,X_2]^\H +R(X_1,X_2)^\V+[X_1^\H,η_2^\V]-[X_2^\H,η_1^\V]+[η_1^\V,η_2^\V]\bigr)\\
&=R(X_1,X_2)+D_{X_1^\H}η_2-D_{X_2^\H}η_1,
\end{align*}
where we have taken into account that $[η_1^\V,η_2^\V]=0$ and $D_{σ^\V}η_i=κ(Tη_i(σ^\V))=0$, $i=1,2$, because $η_1,η_2$ are basic, and Theorem~\ref{covariant.derivative}. 

From Proposition~\ref{derivation.vs.covariant.derivative} we have
\[
D_{[Y_1,Y_2]}σ=\cald^{[Y_1,Y_2]}σ+D_{σ^\V}(κ([Y_1,Y_2])).
\]
From Proposition~\ref{linearization.of.a.vector.field} \thetag{5} we have $[\cald^{Y_1},\cald^{Y_2}]=\cald^{[Y_1,Y_2]}$. Thus for $\Hor$-basic vector fields $Y_1$ and $Y_2$ and any section $σ∈\sec[π]{E}$
\begin{align*}
[D_{Y_1},D_{Y_2}]σ
=[\cald^{Y_1},\cald^{Y_2}]σ
=\cald^{[Y_1,Y_2]}σ
=D_{[Y_1,Y_2]}σ-D_{σ^\V}(κ([Y_1,Y_2])).
\end{align*}
It follows that 
\[
\curv(Y_1,Y_2)σ=-D_{σ^\V}(κ([Y_1,Y_2])),
\]
and taking into account the expression of $κ([Y_1,Y_2])$ above for $Y_i=X_i^\H+η_i^\V$, $i=1,2$, we get 
\[
\curv(Y_1,Y_2)σ=-D_{σ^\V}\Bigl(R(X_1,X_2)+D_{X_1^\H}η_2-D_{X_2^\H}η_1\Bigr),
\]
which proves the statement.
\end{proof}

As particular cases of the expression of the curvature we have
\begin{align*}
\curv(X^\H,Y^\H)σ&=-D_{σ^\V}R(X,Y)\\
\curv(\,η^\V\,,Y^\H)σ&=D_{σ^\V}D_{Y^\H}η\\
\curv(\,η^\V\,,\,ζ^\V\,)σ&=0,
\end{align*}
for $X,Y∈\vectorfields{M}$, $ζ,η∈\sec{E}$ and $σ∈\sec{π^*E}$. 

In Finsler geometry, where $E=\slit{T}M$, the component $θ(σ,η)Y=\curv(\,η^\V\,,Y^\H)σ$ is known as the Berwald curvature tensor, and the component $\operatorname{Rie}(X,Y)σ=\curv(X^\H,Y^\H)σ$ is usually known as the Finsler Riemannian curvature tensor. A detailed study of the curvature tensor of $\ol{\Hor}$ and its implications for the original connection can be found in~\cite{linearization.1}. Here we just mention that the vanishing of the $\ssize\mathsf{VH}$-component of the curvature tensor characterizes basic connections, that is, the linearization of a non linear connection is basic if and only if $θ=0$. 

In what respect to the flatness of the linearized connection we have the following result.
\begin{theorem}
\label{flat}
The following statements are equivalent:
\begin{enumerate}
\item The linearized connection $\ol{\Hor}$ is flat.
\item The set of $\Hor$-basic vector fields is a Lie subalgebra of $\vectorfields{E}$.
\item If $Y_1,Y_2$ are $\Hor$-basic vector fields then $κ([Y_1,Y_2])$ is a basic section.
\item If $X_1,X_2∈\vectorfields{M}$ and $η_1,η_2∈\sec{E}$ then $R(X_1,X_2)+D_{X_1^\H}η_2-D_{X_2^\H}η_1$ is a basic section.
\end{enumerate}
\end{theorem}
\begin{proof}
If $Y_1,Y_2$ are projectable then $[Y_1,Y_2]$ is projectable. Therefore, for $Y_1,Y_2$ two $\Hor$-basic vector fields the bracket $[Y_1,Y_2]$ is $\Hor$-basic if and only if $κ[Y_1,Y_2]$ is basic. This proves the equivalence \thetag{2} $\iff$ \thetag{3}.

In the proof of Theorem~\ref{curvature} we found 
\[
κ([Y_1,Y_2])=R(X_1,X_2)+D_{X_1^\H}η_2-D_{X_2^\H}η_1,
\]
from where it is clear the equivalence \thetag{3} $\iff$ \thetag{4}.
 
\thetag{2}$\implies$\thetag{1}. For two projectable vector fields $Y_1,Y_2$ we know that $[\calf Y_1,\calf Y_2]=\calf [Y_1,Y_2]$. From Theorem~\ref{horizontal=fibre.derivative}, if $Y_1$ $Y_2$ and $[Y_1,Y_2]$ are $\Hor$-basic then this equation reads $[Y_1^{\bar{\H}},Y_2^{\bar{\H}}]=[Y_1,Y_2]^{\bar{\H}}$. Since the horizontal distribution is spanned by $\Hor$-basic vector fields it follows that $\Hor$ is an involutive distribution, and hence it is a flat connection.

\thetag{1}$\implies$\thetag{3} follows from the expression of the curvature in Theorem~\ref{curvature}. If $Y_1,Y_2$ are $\Hor$-basic and the connection $\ol{\Hor}$ is flat then $\curv(Y_1,Y_2)σ=-D_{σ^\V}(κ([Y_1,Y_2]))=0$, for all $σ∈\sec{π^*E}$. It follows that $κ([Y_1,Y_2])$ is a basic section, and hence $[Y_1,Y_2]$ is $\Hor$-basic.
\end{proof}

\section{The jet bundle approach}
\label{section:jets}

An Ehresmann connection on a bundle $\map{τ}{P}{N}$ can be equivalently seen as a section of the first jet bundle $\map{τ_{10}}{J^1τ}{P}$. We refer to the reader to~\cite{Saunders} for the notation and other details in jet bundle theory.

The 1-jet $j^1_nσ$ of a local section $σ$ of $τ$ at a point $n∈N$ can be identified with the tangent map $\map{T_nσ}{T_nN}{T_{σ(n)}P}$. Conversely, given a linear map $\map{ϕ}{T_nN}{T_pP}$ such that $T_pτ∘ϕ=\id_{T_nN}$ there exists a local section $σ$ of $τ$ such that $σ(n)=p$ and $T_nσ=ϕ$. Therefore it make sense to write $j^1_mσ(v)$ for $v∈T_mM$ with the meaning $j^1_mσ(v)\equiv T_mσ(v)$. The projection $τ_{10}$ is defined by $τ_{10}(j^1_nσ)=σ(n)$, and is said to be the target projection. The composition $\map{τ_1=τ○τ_{10}}{J^1τ}{N}$ is said to be the source projection.

If $\map{τ'}{P'}{N}$ is another fiber bundle, a fibered map $\map{ψ}{P}{P'}$ over the identity in $N$ induces a map $\map{j^1ψ}{J^1τ}{J^1τ'}$ fibered over $ψ$ given by $j^1ψ(j^1_nσ)=j^1_n(ψ○σ)$. 

\smallskip

Given a connection on $\map{τ}{P}{N}$ the corresponding horizontal lift $\hlift$ determines
a section $h$ of $τ_{10}$. Indeed, $h(p)=\hlift(p,\,·\,):T_{τ(p)}N→T_pP$ satisfies $Tτ∘h(p)=\id_{T_{τ(p)}N}$ and hence it is a 1-jet with target $p$. Conversely, any section $h$ of $τ_{10}$ defines a connection on $\map{τ}{P}{N}$ by means of the horizontal lift $\hlift(p,v)=h(p)(v)$. Therefore a connection on $\map{τ}{P}{N}$ is but a section of $\map{τ_{10}}{J^1τ}{P}$.

When $\map{τ}{P}{N}$ is a vector bundle then the fiber bundle $\map{τ_1}{J^1τ}{N}$ inherits a vector bundle structure with the operations
\[
j^1_nσ+j^1_nη=j^1_n(σ+η)\quad\text{and}\quad λ·j^1_nσ=j^1_n(λσ),
\]
for $σ,η∈\sec{τ}$ and $λ∈ℝ$. The target map $\map{τ_{10}}{J^1τ}{P}$ is a vector bundle morphism over the identity in $M$. A connection defined by a section $h$ of $τ_{10}$ is a linear connection if $h$ a linear section of $τ_{10}$, or in other words, if $h$ is a vector bundle morphism from $\map{τ}{P}{N}$ to $\map{τ_1}{J^1τ}{N}$ over the identity in $N$.

\medskip

In our concrete case, given a non linear connection $h∈\sec{π_{10}}$ on the vector bundle $\map{π}{E}{M}$ we know that the linearization of $h$ is a linear connection on the vector bundle $\map{π^*π}{π^*E}{E}$, and hence it is given by a section $\bar{h}∈\sec{(π^*π)_{10}}$ of the first jet bundle $J^1(π^*π)$ of sections of $π^*π$. In this section we provide a direct construction of $\bar{h}$ by using the geometry of jet bundles.

\smallskip

\begin{wrapfigure}{l}{0.36\textwidth}
\hskip-15pt
$
\xymatrix{%
&E\times_ME\ar@{<->}[d]_{\simeq}\ar[r]^-{\ul{\calf h}}&J^1\mathsf{p}\ar@{<->}[d]^{\simeq}\\
&π^*E\ar[d]_{π^*π}\ar[r]^-{\calf h}&π_1^*(J^1π)\ar[d]^{\pr_1}\\
&E\ar[d]_{π}\ar[r]^{h}&J^1π\ar[d]^{π_1}\\
&M\ar[r]^{\id}&M
}
$
\end{wrapfigure}
We can consider $h$ as a (generally non linear) bundle map $\map{h}{E}{J^1π}$ from the vector bundle $\map{π}{E}{M}$ to the vector bundle $\map{π_1}{J^1π}{M}$ fibered over the identity in $M$. Thus we can construct the fibre derivative obtaining a map $\map{\calf h}{π^*E}{π_1^*(J^1π)}$. 

If we consider the bundle $\map{\mathsf{p}}{E\times_ME}{M}$, the manifold $π_1^*(J^1π)$ can be identified with $J^1\mathsf{p}$ by identifying $(j^1_mσ,j^1_mη)\simeq j^1_m(σ,η)$. The projection onto the first factor $(j^1_mσ,j^1_mη)\mapsto j^1_mσ$ is then identified with the first jet prolongation $j^1\pr_1$ of the projection $\map{\pr_1}{E\times_ME}{E}$ onto the first factor. Thus we obtain a map $\map{\ul{\calf h}}{E\times_ME}{J^1\mathsf{p}}$. Taking into account that $π_{10}∘h=\id_E$ we have that $\calf π_{10}∘\calf h=\calf \id_E=\id_{π^*E}$. Moreover $\calf π_{10}=\mathsf{p}_{10}$, so that we have proved that $\ul{\calf h}$ is a section of $\mathsf{p}_{10}$ and hence a connection on $\map{\mathsf{p}}{E\times_ME}{M}$. 

However this procedure does not produce a connection on $\map{π^*π}{π^*E}{E}$ but a connection on the bundle $\map{\mathsf{p}}{E\times_ME}{M}$. In order to obtain a connection on $π^*E$ we need a further step.

\paragraph{The canonical immersion}
As we know, a section of $E$ determines a section of $π^*E$, which up to this point has been denoted by the same symbol. It is now convenient to distinguish between them. Consider the map $\map{\imath}{\sec{π}}{\sec{π^*π}}$ given by 
\[
\imath(σ)(a)=\bigl(a,σ(π(a)\bigr),\qquad\text{for all $a∈E$.}
\]
The map $\imath$ induces a canonical immersion $\map{Υ}{π^*(J^1π)}{J^1(π^*π)}$ by means of
\[
Υ(a,j^1_mσ)=j^1_a\bigl(\imath(σ)\bigr).
\]
Interpreting a 1-jet with source $m$ and target $b$ as a linear map $\map{ϕ}{T_mM}{T_bE}$ the expression of $Υ$ is simply
\[
Υ(a,ϕ)=\bigl(\id_{T_aE},ϕ∘T_aπ\bigr).
\]

The canonical immersion is a vector bundle map over the identity in $E$ and, with respect to the target projection, it satisfies 
\[
(π^*π)_{10}∘Υ=(\pr_1,π_{10}∘\pr_2),
\]
or in more explicit terms $(π^*π)_{10}\bigl(Υ(a,j^1_mσ)\bigr) = (a,σ(m))$, where $m=π(a)$.

\smallskip

Basic and semibasic connections can be characterized in terms of the canonical immersion $Υ$. 
\begin{proposition}
\label{characterization.semibasic}
Let $h∈\sec{(π^*π)_{10}}$ be a connection on $π^*E$.
\begin{itemize}
\item $h$ is basic if and only if there exists a connection $\ul{h}∈\sec{π_{10}}$ on $E$ such that $h(a,b)=Υ\bigl(a,\ul{h}(b)\bigr)$ for all $(a,b)∈π^*E$.  
\item $h$ is semibasic if and only if $\Im(h)⊂\Im(Υ)$.
\end{itemize}
\end{proposition}
\begin{proof}
Let $\hlift$ the horizontal lifting associated to $h$, i.e. $\hlift((a,b),w)=h(a,b)(w)$. If $\hlift$ is basic then there exists a connection $\ul{h}∈\sec{π_{10}}$ such that $\hlift((a,b),w)=\bigl(w,ξ^{\ul{\H}}(b,Tπ(w))\bigr)$, where $ξ^{\ul{\H}}$ is the horizontal lift associated to $\ul{h}$. In terms of $h$ and $\ul{h}$ this equation reads $h(a,b)(w)=\bigl(w,\ul{h}(b)(Tπ(w))\bigr)$, or in other words $h(a,b)=\bigl(\id_{T_aE},\ul{h}(b)○T_aπ\bigr)=Υ(a,\ul{h}(b))$. 

Conversely, if $h$ and $\ol{h}$ are related by $h(a,b)=Υ\bigl(a,\ul{h}(b)\bigr)$, then for every $w∈T_aE$ we have $\hlift\bigl((a,b),w\bigr)=h(a,b)(w)=Υ\bigl(a,\ul{h}(b)\bigr)(w)=\bigl(w,\ul{h}(b)(w)\bigr)=\bigl(w,ξ^{\ul{\H}}(b,w)\bigr)$, so that the connection $h$ is basic (and the pullback of the connection $\ul{h}$).

Let $h∈\sec{(π^*π)_{10}}$ be a semibasic connection on $π^*E$, so that $h(a,b)(w)=(w,0_b)$ whenever $w∈\ver[a]{π}$.  It follows that $h(a,b)$ is of the form $h(a,b)=(\id_{T_aE},Φ)$ for some linear map $\map{Φ}{T_aE}{T_bE}$ such that $T_bπ○Φ=T_aπ$ and $Φ(w)=0$ for all $w∈\ver[a]{π}$. This implies that $\ver[a]{π}=\ker(Φ)$: obviously $\ver[a]{π}⊂\ker(Φ)$ and if $w∈\ker(Φ)$ then $T_aπ(w)=T_bπ○Φ(w)=0$, so that $w∈\ver[a]{π}$. Defining the linear map $\map{ϕ}{T_mM}{T_bE}$ by $ϕ(v)=Φ(\tilde{v})$ where $\tilde{v}$ is any vector in $T_aE$ such that $T_aπ(\tilde{v})=v$, we have that $T_bπ○ϕ=\id_{T_mM}$ and $Φ$ factorizes as $Φ=ϕ○T_aπ$. Thus $ϕ$ is a 1-jet and $h(a,b)=(\id_{T_aE},ϕ○T_aπ)=Υ(a,ϕ)$.

Conversely, let $h$ be a connection on $π^*π$ such that $\Im(h)⊂\Im(Υ)$. For every $(a,b)∈π^*E$ there exists a map $\map{ϕ_b}{T_mM}{T_bE}$ such that $h(a,b)=Υ(a,ϕ_b)=(\id_{T_aE},ϕ_b○T_aπ)$. If $w∈T_aE$ is vertical then $\hlift\bigl((a,b),w\bigr)=h(a,b)(w)=\bigl(w,ϕ_b(Tπ(w))\bigr)=(w,0_b)$. It follows that $h$ is semibasic.
\end{proof}

From the canonical immersion it is convenient to construct the map $\map{Ψ}{π_1^*(J^1π)}{J^1(π^*π)}$ given by $Ψ(j^1_mσ,j^1_mη)=Υ(σ(m),j^1_mη)$. The map $Ψ$ is linear and satisfies $(π^*π)_{10}∘Ψ=(π_{10},π_{10})$. Proposition~\ref{characterization.semibasic} can be re-stated in terms of $Ψ$ as follows: a connection $h$ is basic if and only if  there exists a connection $\ul{h}$ on $E$ such that $h(a,b)=Ψ(\ul{h}(a),\ul{h}(b))$; a connection $h$ is semibasic if and only if $\Im(h)⊂\Im(Ψ)$.

With the help of the map $Ψ$ we can easily obtain the linearization of the connection $h$.

\begin{theorem}
\label{linearization.jet.pullback.version}
Let $h∈\sec{π_{10}}$ be a non linear connection on $E$. The map $\bar{h}=Ψ○\calf h∈\sec{(π^*π)_{10}}$ is the section corresponding to the linearization of the connection $h$.
\end{theorem}
\begin{proof}
Using the linearity of $π_{10}$ we have that 
\begin{align*}
π_{10}\bigl(\calf_a h(b)\bigr)
&=π_{10}\Bigl(\lim_{s→0}\frac{1}{s}[h(a+sb)-h(a)]\Bigr)\\
&=\lim_{s→0}\frac{1}{s}\bigl[π_{10}\bigl(h(a+sb)\bigr)-π_{10}\bigl(h(a)\bigr)\bigr]\\
&=\lim_{s→0}\frac{1}{s}[(a+sb)-a]\\
&=b.
\end{align*}
It follows that the map $\bar{h}$ is a section of $(π^*π)_{10}$:
\[
(π^*π)_{10}\bigl(\bar{h}(a,b)\bigr)=(π^*π)_{10}\bigl(Ψ(\calf h(a,b))\bigr)=\bigl(π_{10}(h(a)), π_{10}(\calf_ah(b))\bigr)=(a,b).
\]
Moreover, the connection $\bar{h}$ is semibasic by construction (and Proposition~\ref{characterization.semibasic}) and a first order natural prolongation of $h$. It follows from Proposition~\ref{natural+=linearization} that $\bar{h}$ is the linearization of $h$.

Alternatively, one can easily check that $\bar{h}$ is linear, because $Ψ$ and $\calf_a h$ are both linear maps.
\end{proof}

Notice that the map $B(a,b)$ associated to the linearization is given in this context by $B(a,b)=\calf_a h(b)○T_aπ$.

\begin{remark}
A construction for a kind of linearization of a non linear connection is given in~\cite[\S 2.7]{CrSa.book}. However, the result of the construction in~\cite{CrSa.book} is not a connection on $π^*E$, but a connection on $\map{\mathsf{p}}{E\times_ME}{E}$ (misleadingly identified there with the pullback bundle) which is identical to the connection $\ul{\calf h}$. The final step in~\cite{CrSa.book}, using the map $\map{\imath_0}{J^1π}{J^1\mathsf{p}}$, $\imath_0(j^1_mσ)=j^1_m(σ,0_π)$, produces a map $(a,b)\mapsto \imath_0(\calf_ah(b))$ which projects to $(a,0_m)$, and hence it is not a connection on the pullback bundle $π^*E$.
\end{remark}

\subsection*{Construction on the vertical bundle}

The pullback bundle $π^*E$ is isomorphic to the vertical bundle $\ver{π}$. Hence a construction of the linearized connection in the vertical bundle is possible. In this subsection we provide such construction, that is, given a non linear connection $\map{h}{E}{J^1π}$ on $\map{π}{E}{M}$ we will define a linear connection $\map{\bar{h}_\V}{\ver{π}}{J^1τ_E^\V}$, on the vector bundle $\map{τ_E^\V}{\ver{π}}{E}$ which corresponds to $\bar{h}$.

\smallskip

Consider the vector bundle $\map{\mathsf{p}_\V}{\ver{π}}{M}$. We recall that there exists a canonical isomorphism $\map{\jvj}{\ver{π_1}}{J^1\mathsf{p}_\V}$ defined by 
\[
\jvj\Bigl(\frac{d}{ds}j^1_mσ_s\at{s=0}\Bigr)=j^1_m\Bigl(\frac{d}{ds}σ_s\at{s=0}\Bigr),
\]
where $σ_s$ is a curve in the set of sections of $E$. See~\cite[\S 5]{MaMo} for the details.

Consider the connection $\map{h}{E}{J^1π}$. Given a vertical vector $z=\vlift_π(a,b)∈\Ver_a{π}$ we consider the curve $s\mapsto h(a+sb)∈J^1π$ and we take the tangent vector at $s=0$. This is a vector on $J^1π$ vertical over $M$ (due to $π_1○h=π$) to whom we may apply the diffeomorphism $\jvj$ obtaining an element in $J^1\mathsf{p}_\V$,
\[
\calv h(z)=\jvj\Bigl(\frac{d}{ds}h(a+sb)\at{s=0}\Bigr)∈J^1\mathsf{p}_\V,\qquad z=\vlift_π(a,b).
\]
It is easy to see that $\calv h$ is a linear connection on the bundle $\ver{π}\to M$.

\begin{remark}
Under the identification of $E\times_ME→M$ with $\ver{π}→M$ given by $\vlift_π$, the connection $\ul{\calf h}$ corresponds to the unique natural prolongation $\calv h$ to a connection on $\ver{π}→M$ defined in~\cite[\S 31.1]{NatOp}. Thus the connection $\ul{\calf h}$ that we have obtained is the unique natural prolongation of $h$ to a connection on  
$\map{\mathsf{p}}{E\times_ME}{M}$.
\end{remark}

A section $V$ of $\map{\mathsf{p}_\V}{\ver{π}}{M}$ defines a section $\imath_\V(V)$ of $\map{τ_E^\V}{\ver{π}}{E}$ by means of 
\[
\imath_\V(V)(a)=\vlift_π\bigl(a,ν_π(V(π(a)))\bigr),\qquad\text{for all $a∈E$.}
\]
This relation induces a map $\map{Ψ_\V}{J^1\mathsf{p}_\V}{J^1τ_E^\V}$ given by
\[
Ψ_\V(j^1_mV)=j^1_{τ_E(V(m))}\bigl(\imath_\V(V)\bigr).
\]
The map $Ψ_\V$ satisfies $(τ_E^\V)_{10}∘Ψ_\V=(\mathsf{p}_\V)_{10}$, or in other words $(τ_E^\V)_{10}∘Ψ_\V(j^1_mV) = V(m)$ for all $j^1_mV∈J^1\mathsf{p}_\V$.

In the following result $j^1\vlift_{π,E}$ denotes the 1-jet prolongation of the map $\vlift_π$ fibered over the identity in~$E$, and $j^1\vlift_{π,M}$ denotes the 1-jet prolongation of the map $\vlift_π$ fibered over the identity in~$M$.

\begin{theorem}
\label{linearization.jet.vertical.version}
Let $h∈\sec{π_{10}}$ be a non linear connection on $E$. The map $\map{\bar{h}_\V}{\ver{π}}{J^1τ_E^\V}$ defined by $\bar{h}_\V=Ψ_\V∘\calv h$ is a linear connection on the vertical bundle $\map{τ_E^\V}{\ver{π}}{E}$ related with the connection $\bar{h}$ by $j^1\vlift_{π,E}∘\bar{h}=\bar{h}_\V∘\vlift_π$.
\end{theorem}
\begin{proof}
The proof is based on the following facts whose proof is omitted. For the canonical isomorphism we have that $\jvj○\vlift_{π_1}=j^1\vlift_{π,M}$, and for the maps $Ψ$ and $Ψ_\V$ we have the relation $Ψ_\V○j^1\vlift_{π,M}=j^1\vlift_{π,E}○Ψ$. Then for all $(a,b)∈π^*E$,
\begin{align*}
\bar{h}_\V\bigl(\vlift_π(a,b)\bigr)
&=Ψ_\V\left(\jvj\Bigl(\frac{d}{ds}h(a+sb)\at{s=0}\Bigr)\right)\\
&=Ψ_\V\left(\jvj\bigl(\vlift_{π_1}(\calf h(a,b))\bigr)\right)\\
&=Ψ_\V\left(j^1\vlift_{π,M}\bigl(\calf h(a,b)\bigr)\right)\\
&=j^1\vlift_{π,E}\left(Ψ\bigl(\calf h(a,b)\bigr)\right)\\
&=j^1\vlift_{π,E}\bigl(\bar{h}(a,b)\bigr).
\end{align*}
As $\bar{h}$ is a linear connection on $π^*π$ and $\vlift_π$ is an isomorphism of vector bundles it follows that $\bar{h}_\V$ is a linear connection on $τ^\V_E$.
\end{proof}

\subsection*{Local coordinate expressions}

In the finite dimensional case, consider a system of local coordinates $(x^i,y^A)$ on~$E$, $(x^i,y^A,z^A)$ on $π^*E$, $(x^i,y^A,y^A_i)$ on $J^1π$, and $(x^i,y^A,z^A,z^A_i,z^A_B)$ in $J^1(π^*π)$. The coordinate expression of $Υ$ is 
\[
Υ\bigl((x^i,y^A),(x^i,z^A,z^A_i)\bigr)=(x^i,y^A,z^A,z^A_i,0),
\]
and hence the map $Ψ$ is
\[
Ψ\bigl((x^i,y^A,y^A_i),(x^i,z^A,z^A_i)\bigr)=(x^i,y^A,z^A,z^A_i,0).
\]
If $h$ is given locally by $h(x^i,y^A)=(x^i,y^A,-Γ^A_i(x,y))$ then
\[
\calf h\bigl((x^i,y^A),(x^i,z^A)\bigr)
=\Bigl(\bigl(x^i,y^A,-Γ^A_i(x,y)\bigr),\bigl(x^i,z^A, -Γ^A_{iB}(x,y)z^B\bigr)\Bigr).
\]
Thus the composition of both maps gives
\begin{align*}
\bar{h}(x^i,y^A,z^A)
&=(x^i,y^A,z^A,-Γ^A_{iB}(x,y)z^B,0).
\end{align*}

In the case of the vertical bundle, we can take coordinates $(x^i,y^A,z^A)$ in $\ver{π}$, $(x^i,y^A,z^A,y^A_i,z^A_i)$ in $J^1\mathsf{p}_\V$, $(x^i,y^A,y^A_i,\dot{y}^A,\dot{y}^A_i)$ in $\ver{π_1}$, and $(x^i,y^A,z^A,z^A_i,z^A_B)$ in $J^1τ_E^\V$, the local expression of $\jvj$ is 
\[
\jvj(x^i,y^A,y^A_i,\dot{y}^A,\dot{y}^A_i)=(x^i,y^A,\dot{y}^A,y^A_i,\dot{y}^A_i).
\]
Due to the choice of the names of the coordinates, the local expression of $\calf h$ is exactly equal to the local expression of $\calv h$, the local expression of $Ψ_\V$ is exactly equal to the local expression of $Ψ$, and the local expression of $\bar{h}_\V$ is exactly equal to the local expression of $\bar{h}$.


\let\sep\newline

\end{document}